\documentclass[a4paper,11pt]{article}

\usepackage{amsmath,amsfonts,amssymb,bm}

\usepackage{cancel}
\usepackage{tikz-cd}
\usepackage{float}
\usepackage[all]{xy}
\usepackage{ulem}

\usepackage[latin1]{inputenc}
\usepackage{diagbox,pdflscape,multirow}
\usetikzlibrary{positioning}

\usepackage{caption}
\usepackage[numbers,sort&compress]{natbib}
\usepackage{graphicx}

\usepackage{hyperref}
\hypersetup{
  colorlinks   =  true,
    linkcolor    = cyan,
    citecolor    = red,
     urlcolor	=magenta,     
}

\usepackage{amsthm}

\newtheorem{definition}{Definition}[section]
\newtheorem{lemma}[definition]{Lemma}
\newtheorem{theorem}[definition]{Theorem}
\newtheorem{proposition}[definition]{Proposition}
\newtheorem{corollary}[definition]{Corollary}
\newtheorem{remark}[definition]{Remark}
\newtheorem{example}[definition]{Example}

\numberwithin{equation}{section}
\newcommand{\red}[1]{{\color{red}#1}}
\newcommand{\blue}[1]{{\color{blue}#1}}

\def \g{\mathfrak{g}}
\def \R{\mathbb{R}}

\def \p {\eta}
\def \M {\mathcal{M}}
\def \h{\mathfrak{h}}
\def \l{\mathfrak{l}}
\def \V{\mathcal{V}_0}
\def \H{\mathcal{H}}
\def \X{\mathfrak{X}}
\def \W {\H_\nu^\Pi}

\newcommand{\lvec}[1]{\overleftarrow{#1}}
\newcommand{\rvec}[1]{\overrightarrow{#1}}

\DeclareMathOperator{\Tr}{Tr}

\def\p {\eta}

\begin{document}

\title{Unimodularity and invariant volume forms for Hamiltonian dynamics on coisotropic Poisson homogeneous spaces}

\author{I. Gutierrez-Sagredo$^{1,2}$, D. Iglesias Ponte$^{2,3}$, J. C. Marrero$^{2,3}$, E. Padr\'on$^{2,3}$}
\date{}

\maketitle

\vspace{-20pt}
\begin{center}
{\small\it  $\;^1$ Departamento de Matem\'aticas y Computaci\'on, Universidad de Burgos, 09001 Burgos, Spain}

{\small\it $\;^2$Departamento de Matem\'aticas, Estad\'{\i}stica e Investigaci\'on Operativa}\\{\small\it University of La Laguna, La Laguna, Spain}
\\[5pt]

{\small\it $\;^3$ULL-CSIC Geometr\'{\i}a Diferencial y Mec\'anica Geom\'etrica and Instituto de Matem\'aticas y Aplicaciones IMAULL-University of La Laguna, La Laguna, Spain}
\\[5pt]
e-mail: {\small \href{mailto:igsagredo@ubu.es}{igsagredo@ubu.es}, \href{mailto:idiglesia@ull.edu.es}{diglesia@ull.edu.es}, \href{mailto:jcmarrer@ull.edu.es}{jcmarrer@ull.edu.es}, \href{mailto:mepadron@ull.edu.es}{mepadron@ull.edu.es} }

\end{center}

\begin{abstract} 
In this paper, we introduce a notion of  multiplicative unimodularity for a coisotropic Poisson homogeneous space. Then, we discuss the unimodularity and the multiplicative unimodularity for these spaces and the existence of an invariant volume form for explicit Hamiltonian systems on such spaces. Several interesting examples illustrating the theoretical results are also presented. 

\medskip 



\end{abstract}

\tableofcontents

\section{Introduction}
\label{sec:intro}
\subsection{Invariant volume forms for Hamiltonian systems on Poisson-Lie groups}\label{section1.1}

When studying a dynamical system  $X\in \mathfrak{X}(M)$ on a manifold $M$ of dimension $n$, it is natural to look for conserved quantities to study its integrability. For instance, if there is a family $\{f_1,\ldots ,f_{n-1}\}$ of independent first integrals ($X(f_i)=0$) then
it is possible to obtain the orbits of $X$. 
A different possibility is to look, not only for first 
integrals, but also for an invariant volume form $\Omega$: if there exist $\{f_1,\ldots ,f_{n-2}\}$ independent first integrals, the existence of such a volume form guarantees that the dynamical system can be integrated by quadratures (see, for instance, \cite{Ko}). 

In the special case of a Hamiltonian system  on  a symplectic manifold of dimension $2n$, 
there exists the notion of Liouville integrability. The system is Liouville integrable if one can find $n$
functionally independent first integrals which are Poisson commuting (i.e., they are in involution with respect to the Poisson bracket).
On the other hand, with respect to invariant volume forms, Liouville's theorem states that 
the flow of the Hamiltonian vector field on a symplectic manifold preserves the symplectic volume (see, for instance, \cite{AM,Ar}).
  
More generally, when the dynamical system is Hamiltonian with respect to a Poisson structure (not necessarily symplectic), Liouville's theorem can be applied if we restrict to the symplectic leaves and consider the symplectic structures on them. However, in general, the symplectic foliation is hard to compute, and thus it is difficult to obtain invariant volume forms on the leaves. Taking this into account it seems more reasonable to look for an invariant volume form on the whole manifold. 

Now, given an orientable Poisson manifold $(M,\Pi )$, 
the modular class is a first order cohomology 
class in the Poisson  cohomology complex of $M$ 
defined by a Poisson vector field, which measures 
the existence of an invariant volume form for the 
flows of all Hamiltonian vector fields (see 
\cite{W96}). If the modular class vanishes then $(M,\Pi )$ is said to be unimodular.
By the aforementioned Liouville's theorem, the modular class of a symplectic manifold is zero. 
 
 Using that Lie algebroids are the dual objects to linear Poisson structures on a vector bundle, one can also define the modular class of a Lie algebroid \cite{ELW, W96}. In fact, the relation between the unimodularity of a Lie algebroid and the existence of invariant volume forms for the hamiltonian dynamics on the dual bundle was discussed in \cite{M}.

More recently, in \cite{gutierrez2023unimodularity}, we studied the existence of invariant volume forms for Hamiltonian systems on Poisson-Lie groups. Recall that a Poisson-Lie group is a Lie group $G$ endowed with a Poisson structure such that the multiplication is a Poisson map \cite{Dr}.  Poisson-Lie groups are in one-to-one correspondence with the so-called Lie bialgebras $(\mathfrak{g},\mathfrak{g}^*)$, which are compatible pairs of Lie algebras in duality.  In \cite{ELW} it has been given a description of the modular vector field of a Poisson-Lie group $G$ with respect to a left-invariant volume form using the modular character of the dual Lie algebra of $G$. From this description, in \cite{gutierrez2023unimodularity} we proved 

{\it Given a Hamiltonian function $h$ on a Lie group $G$ endowed with a unimodular Poisson-Lie structure $\Pi _G$ (that is, the dual Lie algebra of $G$ is unimodular), then the Hamiltonian vector field of $h$, $X_{h}^{\Pi_G}$, preserves a multiple of any left-invariant volume form on $G$. Conversely, if the identity element of $G$ is a nondegenerate singularity of a Hamiltonian function  $h$ and $X_{h}^{\Pi_G}$ preserves a volume form then the Poisson-Lie structure $\Pi_G$ is unimodular.}    

As a consequence, we recovered a result in 
 \cite{K88} about the relation between the  unimodularity of a Lie algebra $\g$ and the preservation of a volume form by 
 the flow associated with  a Hamiltonian function of kinetic type on the dual space $\mathfrak{g}^*$.
 
 Modular classes have also been introduced and studied for more general objects as skew algebroids \cite{Gr} and the dual objects of the latter: linear almost Poisson structures on vector bundles (see \cite{FGM}). In fact, in \cite{FGM} the relation between the unimodulatity of a linear almost Poisson structure and the existence of invariant volume forms for nonholonomic mechanical systems was also discussed.

\subsection{ Poisson homogeneous spaces}\label{Section1.2}
An important class of Poisson manifolds related to Poisson-Lie groups is that of Poisson homogeneous spaces, which were introduced by Drinfeld \cite{Dr2} as semi-classical limits of
homogeneous spaces of quantum groups. Given a Poisson-Lie group $(G,\Pi _G)$ and a closed Lie subgroup $H,$ a homogeneous Poisson structure on $G/H$ is a Poisson structure
$\Pi$ on $G/H$ such that the natural transitive action 
\[
\begin{array}{ccc}
(G, \Pi _G) \times (G/H, \Pi ) &\to & (G/H, \Pi ),\\ [8pt]
 (g_1, g_2H) & \mapsto & g_1g_2H,
\end{array}
\]
is a Poisson map. 

In this paper, we are interested in a particular class of Poisson homogeneous spaces, namely coisotropic Poisson quotients. For these spaces, the canonical projection $q:(G,\Pi_G)\to (G/H,\Pi)$ is a Poisson epimorphism or, equivalently, the annihilator $\h^0$ of the Lie algebra $\h$ of $H$  is a Lie subalgebra of the dual Lie algebra $\g^*$. 

Coisotropic Poisson quotients play a special role in the description of integrable systems, such as the Toda lattice (see  \cite{Rui} and the references therein). In addition, we remark the following facts: 
\begin{enumerate}
\item[(i)] The explicit  description of the symplectic leaves for a Poisson homogeneous space and the symplectic structure on them may be hard (see 
\cite{Ka}).

\item[(ii)] Similarly to Poisson-Lie groups, a Poisson homogeneous space is integrable and a symplectic groupoid ${\mathcal G}(G/H)$ integrating $G/H$ may be constructed (see \cite{BuIgLu}; see also \cite{BoCiStTa1,BoCiStTa2}). However, the dimension of ${\mathcal G}(G/H)$ is $2\dim(G/H)$ and the integration is not always explicit.
\end{enumerate}

\subsection{Aim and contributions of the paper}
The aim of this paper is to discuss the existence of a volume form which is invariant for a Hamiltonian system on a coisotropic Poisson homogeneous space. For this purpose, one could apply Liouville Theorem to: (i) the restriction of the Hamiltonian system to the symplectic leaves of the homogeneous Poisson structure or (ii) the natural extension of the Hamiltonian system to the symplectic groupoid ${\mathcal G}(G/H)$ integrating $G/H$. However, for reasons (i) and (ii) in Section \ref{Section1.2}, it seems reasonable to study the problem directly in the Poisson homogeneous space $G/H$. In fact, we will discuss the relation between the existence of a volume form on the whole homogeneous space $G/H$
which is invariant for the Hamiltonian system and the unimodularity of the associated Lie algebras $\g,\g^*$ and $\h^0$, extending the results of our previous paper \cite{gutierrez2023unimodularity}. 

A difference, which first arises when dealing with a homogeneous space, is that there is not a natural choice for a volume form, as opposed to the Lie group case, where one may choose left (or right)-invariant volume forms. Now, given a volume form $\nu$ on a homogeneous space $M=G/H$ then
\[
q^*\nu =e^\sigma \lvec{\V},
\]
where $\sigma\in C^\infty (G)$ such that $\sigma (e)=0$, $\V =(q^*\nu )(e)\in \wedge^{top}\h^0$ and $\lvec{\V}$ is the left-invariant form on $G,$ whose value at the identity element is $\V$. In this situation, there are natural compatibility conditions between $\V$ and $d\sigma (e)$ involving the Lie algebra cohomology and the modular characters $\chi_\g$ and $\chi_\h$ of $\g$ and $\h$, respectively. 

In this paper, we are interested in a particular class of volume forms on $G/H$, the so-called semi-invariant volume forms \cite{Mos}, for which the function $\sigma$ is multiplicative. Note that if $\sigma$ identically vanishes, then the volume form is invariant under the transitive $G$-action.

Now, let $G/H$ be a coisotropic Poisson homogeneous space with Poisson structure $\Pi$. Then, we prove (Proposition \ref{prop:modular-class}) that the modular vector field of $\Pi$ with respect to $\nu$ is given by 
\[
\M_\nu^\Pi=q_*\left(-\Pi _G^\sharp (d\sigma ) +  \frac12 \left(  \rvec{\chi_{\g^*}} - \lvec{\chi _{\g^*}}+2 \lvec{x_{\h ^0}} + \Pi_G^\sharp 
 (\rvec{\chi_\g})\right)\right), 
 \]
 where $\Pi_G^\#:T^*G\to TG$ is the vector bundle morphism associated with the Poisson-Lie structure on $G$, $\rvec{\chi_{\g^*}}$ (respectively, $\rvec{\chi_{\g}}$) is the right-invariant vector field (respectively, $1$-form) on $G$ whose value at the identity element is $\chi_{\g^*}\in\g$ (respectively, $\chi_{\g}\in \g^*$), 
 $\lvec{\chi_{\g^*}}$ is the left-invariant vector field on $G$ whose value at the identity element of $G$ is $\chi_{\g^*}\in\g$, and
$\lvec{x_{\h^0}}$ is any left-invariant vector field on $G$ whose value $x_{\h^0}$ at the identity element projects on the modular character ${\chi_{\h^0}}\in (\h^0)^*\cong \g/\h$ of $\h^0.$

 
From this expression, it is clear that if the volume form $\nu$ is semi-invariant (i.e., $\sigma$ is multiplicative) and $\h^0$ is unimodular or, in other words, when the coisotropic homogeneous Poisson structure is multiplicative unimodular (Definition \ref{definition:multiplicative-unimodular}), then the modular class is the $q$-projection of  the vector field on $G$ given by
\[
\H ^{\Pi} _\nu= -\Pi _G^\sharp (d\sigma ) + \frac12 \left(  \rvec{\chi_{\g^*}} - \lvec{\chi _{\g^*}} + \Pi_G^\sharp 
 (\rvec{\chi_\g})\right),
 \]
 which is not only a Poisson vector field but also multiplicative. In fact, one of the main results of our  paper is to characterize infinitesimally when this situation happens. More precisely, the main contributions of the paper are: 
 \begin{itemize}
 \item We describe the nature of an arbitrary volume form on $G/H$ in terms of some geometric objects in the Lie group $G$ (Theorem \ref{thm:volume-form}) and we apply this result to the case when the volume form is semi-invariant (Theorem \ref{2.5}). In particular, we characterize infinitesimally when a homogeneous space admits a semi-invariant volume form (Proposition \ref{cor:caract:semi-invariante}). 
 \item We prove that for a unimodular coisotropic Poisson quotient, the Lie algebra $\h^0$ is unimodular (Theorem \ref{th:uni-perp}).
 \item We obtain an infinitesimal characterization of the multiplicative unimodularity (Theorem \ref{thm:caracterizacion}). 
 \item We prove a version, for the more general case of a coisotropic Poisson quotient, of the main result in \cite{gutierrez2023unimodularity} for Poisson-Lie groups (see Section \ref{section1.1} for the result in \cite{gutierrez2023unimodularity}). Namely, if the Hamiltonian system induced by a $H$-Morse function at $eH$ on a coisotropic Poisson homogeneous space $G/H$ preserves a volume form then $\h^0$ is unimodular (Theorem \ref{th:morseh0}). 
 
 \item Throughout the paper, several interesting examples illustratring the theoretical results are presented. In particular: a) we prove that there do not exist any Poisson homogeneous spaces for $SO(3)$ and $SL(2,\R)$ which is unimodular (Section \ref{sec:examples});  b) on the non-unimodular coisotropic homogeneous Poisson space $S^2\cong SO(3)/S^1$, we present a Hamiltonian system with $H$-Morse Hamiltonian function which preserves a $SO(3)$-invariant volume form (Example \ref{Ex:6.5}) ; c) we prove that the Toda lattice Hamiltonian system on the coisotropic homogeneous quotient $SL(n,\R)/SO(n)$ does not preserve a volume form on the homogeneous space (Example \ref{Ex:6.6})  and d) we exhibit a Hamiltonian system (a compartmental epidemiological model \cite{BBG2020physicaD}) on a multiplicative unimodular coisotropic Poisson quotient, which does not admit invariant volume forms for the transitive action associated with the homogeneous space (Example \ref{Ex:6.7}). 

  \end{itemize}

  \subsection{Organization of the paper}

The paper is organized as follows. We first describe, in Section \ref{sec:volume-forms}, (semi-invariant) volume forms on a homogeneous space $G/H$
and characterize them in terms of geometric objects in the Lie group $G.$

 In Section \ref{sec:PLgroup}, we recall the notion of a Poisson-Lie group and a Poisson homogeneous space, describing the infinitesimal objects which characterize them. Then, in Section \ref{sec:unimodularity}, we prove that for a unimodular coisotropic Poisson quotient $(G/H,\Pi)$, the Lie subalgebra $\h^0$ is unimodular and we obtain an infinitesimal characterization of the multiplicative unimodularity. 
 
 In Section \ref{sec:examples}, we discuss the unimodularity of coisotropic Poisson  quotients of the Poisson-Lie groups $SO(3)$ and $SL(2,\R)$. 
 
 Finally, in Section \ref{sec:preservation}, we analyze the relation between the unimodularity of coisotropic Poisson homogeneous spaces and the preservation of volume forms for Hamiltonian systems on such spaces, showing that the relationship is not as simple as the one for Poisson-Lie groups. We also present several examples illustrating the previous fact.

{{\bf Assumption:} All Lie groups considered in this paper are connected. }

\section{Volume forms on homogeneous spaces}\label{sec:volume-forms}
In this section, we will study volume forms on homogeneous spaces and characterize them in terms of functions on the Lie group associated with the homogeneous space. 

Let $G$ be a connected Lie group with Lie algebra $\g$, $\textrm{dim }G=m$, and $H$ be a closed Lie subgroup of $G$ with Lie algebra $\h$,  $\textrm{dim }H=n$. Let us denote by $\phi:G\times G/H\to G/H$ the transitive action of $G$ on $G/H$ induced by the group structure on $G$, $q\colon G\to G/H$ the quotient projection, $\h^0$ the annihilator of the Lie algebra $\h$ and $e$ the indentity element of $G.$ 

If $\nu$ is a volume form on the homogeneous space $M=G/H$ then
\begin{equation}\label{eq:volume-form}
(q^*\nu)(g)\neq 0,\qquad \forall g\in G.    
\end{equation}
On the other hand, as $q^*\nu$ is a $q$-basic form, it follows that
\begin{equation}\label{eq:q-basic}
\left ( i_{\lvec{X}} ( q^*\nu )\right )(g) =0,\qquad \forall  X\in \h, 
\end{equation}
where $\lvec{X}$ is the left-invariant vector field associated with $X$. 
Therefore, $\V =(q^*\nu )(e)\in \wedge^{m-n}\h^0$, $\V\neq 0$, and, from \eqref{eq:volume-form} and \eqref{eq:q-basic}, we deduce that
\[ q^*\nu =f \lvec{\V},\]
with $f\in C^\infty(G)$, $f(g)\neq 0$, for any $g\in G$ and $f(e)=1$. Here $\lvec{\V}$ is the left-invariant $(m-n)$-form on $G$ associated with $\V$. As $G$ is connected, we may suppose that $f(g)>0$, for all $g\in G$, that is $f=e^\sigma$, where $\sigma\in C^\infty (G)$ such that $\sigma (e)=0$. Thus, we can write
\[
q^*\nu =e^\sigma \lvec{\V}.
\]
We will describe in the next results the properties which are satisfied by $\sigma$ and $\V$.

Along the paper, $\chi_\g$ will denote the modular character of the Lie algebra $\g$. We recall that $\chi_\g\in \g^*$ is the trace of the adjoint action $ad^\g\colon \g\times \g\to \g$.  If $\{X_\alpha\}$ is a basis of $\g$ with dual basis $\{X^\alpha\}$ then
\[
\chi_\g =C_{\alpha\beta}^\beta X^\alpha , 
\]
where $C_{\alpha \beta}^\gamma$ are the structure constants of $\g$.

\begin{lemma}\label{lem:lema0}
Let $\g$ be a Lie algebra and $\h$ be a Lie subalgebra, such that $\textrm{dim }\g=m$ and $\textrm{dim }\h=n$. If $\V\in \wedge^{m-n}\h^0$ and $\V\neq 0$, then
\[
d^\g\V=-\theta _0\wedge \V,
\]
where $\theta_0\in \g^*$ satisfies
\[ (\theta _0)_{|\h}=\chi_\g{}_{|\h}-\chi_\h,\]
$\chi_\g{}_{|\h}$ being the restriction of $\chi_\g$ to $\h$ and $d^\g$ the algebraic differential of $\g.$

In addition, if $d^\g \V=-\theta '_0\wedge \V$ with $\theta '_0\in \g^*$ then 
\[ (\theta '_0)_{|\h}=(\theta _0)_{|\h}.\]
Conversely, if there is an element $\theta_0'$ of $\g^*$ such that $(\theta '_0)_{|\h}=(\theta _0)_{|\h}$, then $d^\g\V=-\theta' _0\wedge \V.$
\end{lemma}
\begin{proof}
Using that $\V\neq 0$, it follows that
\[\V=Y^1\wedge \ldots \wedge Y^{m-n},\]
where $Y^1,\ldots ,Y^{m-n}\in \h^0$. It is clear that $\{Y^1,\ldots ,Y^{m-n}\}$ are linearly independent and we can consider a basic of  $\g^*$ of the form
\[
\{ X^1,\ldots, X^n,Y^1,\ldots ,Y^{m-n}\}.
\]
Let $\{ X_1,\ldots, X_n,Y_1,\ldots ,Y_{m-n}\}$ be the dual basis on $\g$. Since $\{X_1,\ldots, X_n\}$ is a basis of $\h$ and $\V\in \wedge^{m-n}\h^0$, it follows that 
\[
d^\g\V(Z_0,Z_1,\ldots, Z_{m-n})=0, \qquad \mbox{ if $Z_i,Z_j \in \h,$ with $i\not=j.$}
\]
This implies that
\begin{equation}\label{eq:wedge-chi0}
d^\g \V = \Big ( \sum _{i} (d^\g \V) (X_i,Y_1,\ldots ,Y_{m-n})X^i\Big  )\wedge \V.  
\end{equation}
So, we can take $\theta_0=-\displaystyle\sum _{i} (d^\g \V) (X_i,Y_1,\ldots ,Y_{m-n})X^i\in \g^*$ and  we must prove that
\[
-\Big ( \sum _{i} (d^\g \V) (X_i,Y_1,\ldots ,Y_{m-n})X^i\Big ) _{|\h}=
\chi _\g{}_{|\h}-\chi _\h .
\]
Now, it follows that
\begin{eqnarray*}
(d^\g \V) (X_i,Y_1,\ldots ,Y_{m-n})&=&\sum_{j} (-1)^j 
\V ( [X_i,Y_j],Y_1,\ldots ,\hat{Y}_j,\ldots ,Y_{m-n})\\
&=& - \sum_{j} \V ( Y_1,\ldots ,[X_i,Y_j],\ldots ,Y_{m-n})\\
&=& - \sum_{j} Y^j ([X_i,Y_j])=-\left ( \chi_\g(X_i)-\chi_\h(X_i) \right ),
\end{eqnarray*}
where we have used that $\V(Y_1,\ldots ,Y_{m-n})=1$ and the definition of  $d^\g.$ Thus, from \eqref{eq:wedge-chi0}, we deduce that
\[ d^\g\V =-\theta _0\wedge \V,\]
with $\theta _0(X_i)=\chi_\g(X_i)-\chi_\h(X_i),$ for all $X_i\in \h$.

In addition, if we consider another $\theta '_0\in \g^*$ such that
\[ d^\g \V=-\theta '_0\wedge \V\]
then, since $\V\in \wedge^{m-n}\h^0$,  it is clear that  $\theta_0\wedge \V=\theta '_0\wedge \V$ if and only if $(\theta '_0)_{|\h}=(\theta _0)_{|\h}=\chi_\g{}_{|\h}-\chi_\h$.
\end{proof}

Now, we can describe the volume forms on the homogeneous space $M=G/H$. 

\begin{theorem}\label{thm:volume-form}
Let $\nu$ be a $(m-n)$-form on the homogeneous space $M=G/H$ and $q\colon G\to G/H$ be the canonical projection. If $\nu$ is a volume form on $G/H$ then there is a function $\sigma\in C^\infty(G)$ and $\V\in \wedge^{m-n}\h^0$ such that 
\begin{equation}\label{eq:Key-formula}
    q^*\nu =e^\sigma \lvec{\V},
\end{equation}
with 
\begin{itemize}
    \item[i)] $\sigma(e)=0$, $\lvec{X}(\sigma)=X(\sigma)$, for all $X\in\h$, and $d\sigma (e)_{|\h}=\chi_\g{}_{|\h}-\chi_\h$;
    \item[ii)]  $\V\neq 0$ and $d^\g \V=-d\sigma (e)\wedge \V$.
\end{itemize}
Conversely, if $\sigma\in C^\infty(G)$ and $\V\in \wedge^{m-n}\h^0$, satisfy \textit{i)} and \textit{ii)} then there exists a unique volume form $\nu$ on $G/H$ which is characterized by condition \eqref{eq:Key-formula}.
\end{theorem}
\begin{proof}
If $\nu$ is a volume form on $G/H$, we have already seen that 
\[  q^*\nu =e^\sigma \lvec{\V},\] 
where $\sigma\in C^\infty(G), \sigma(e)=0$ and $\V\in \wedge^{m-n}\h^0,\V\not=0.$

Moreover, 
\[0=q^*(d \nu)=d(q^*\nu )=e^\sigma (d\sigma\wedge \lvec{\V}+\lvec{d^\g\V}). \]
Thus, we have that $\lvec{d^\g\V}=-d\sigma\wedge \lvec{\V}$. 
In particular,
\[
d^\g\V=-d\sigma(e) \wedge \V.
\]
Therefore, from Lemma \ref{lem:lema0}, we obtain that
\[ d\sigma (e)_{|\h}= \chi_\g{}_{|\h} -\chi_\h.  \]
In addition,  using that $e^\sigma\lvec{\V}$ is a $q$-basic form, we have that, for every $X\in \h$,
\[
0=\mathcal{L}_{\lvec{X}}(e^\sigma \lvec{\V})=e^\sigma \Big ( 
\lvec{X}(\sigma )\lvec{\V}+\mathcal{L}_{\lvec{X}} \lvec{\V}
\Big ).
\]
This implies that 
\[ \lvec{X}(\sigma )\lvec{\V}=-\mathcal{L}_{\lvec{X}} \lvec{\V}, \qquad  \forall X\in \h.\]
So, $\lvec{X}(\sigma )\lvec{\V}$ must be a left-invariant $(m-n)$-form and, thus, $\lvec{X}(\sigma)$ is constant.

Conversely, suppose that $\sigma\in C^\infty (G)$ and $\V\in \wedge^{m-n}\h^0$ satisfy \textit{i)} and \textit{ii)}. Then, we can consider the $(m-n)$-form on $G$ given by $e^\sigma \lvec{\V}$. We must prove that $e^\sigma\lvec{\V}$ is $q$-basic, that is, 
\begin{equation}\label{eq:caract-q-basic}
    i_{\lvec{X}}(e^\sigma\lvec{\V} )=0 \quad \text{ and } \quad  \mathcal{L}_{\lvec{X}} (e^\sigma\lvec{\V} )=0, \qquad \forall X\in \h.
\end{equation}
The first equation is trivial because $\V\in \wedge^{m-n}\h^0$, $X\in \h$ and $i_{\lvec{X}}\lvec{\V}=\lvec{i_X\V}=0$. 

On the other hand, using the hypotheses \textit{i)} and \textit{ii)}, 
\begin{eqnarray*}
    \mathcal{L}_{\lvec{X}} (e^\sigma\lvec{\V} )&=&i_{\lvec{X}}(d(e^\sigma\lvec{\V} )) = i_{\lvec{X}} (e^\sigma (d\sigma\wedge \lvec{\V}+\lvec{d^\g\V}))\\
    &=&  i_{\lvec{X}} (e^\sigma (d\sigma\wedge \lvec{\V}-
    \lvec{d\sigma(e)}\wedge \lvec{\V} ))=e^\sigma (X(\sigma )\lvec{\V}-
    X(\sigma )\lvec{\V})=0.
\end{eqnarray*}

Since both equations in \eqref{eq:caract-q-basic} hold, we deduce that there exists a unique $(m-n)$-form $\nu$ on $G/H$ such that $q^*\nu=e^\sigma \lvec{\V}$. Finally, using that $e^{\sigma(g)}\lvec{\V}(g)\neq 0$, for any $g\in G$, and the fact that $q^*$ is an injective morphism, we conclude that $\nu$ is a volume form on $G/H$.
\end{proof}
Next, we recall a natural kind of volume forms on  homogeneous spaces.
\begin{definition}[\cite{Mos}]\label{def:semi-inv}
Let $G$ be a connected Lie group of dimension $m$ and $H$ be a closed Lie subgroup of $G$ of dimension $n$. A volume form $\nu$ on $G/H$ is said to be \textit{semi-invariant} if 
\[
\phi _g^*\nu=c(g)\nu , \qquad \forall g\in G,
\]
with $c\colon G\to \R^+$ a smooth function, where $\phi\colon G\times G/H\to G/H$ is the transitive action of $G$ on $G/H$. If $c(g)= 1,$ for all $g\in G$, then $\nu$ is $G$-invariant.
\end{definition}
\begin{remark}
Note that, from Definition \ref{def:semi-inv}, it is deduced that
$c\colon G\to (\R^+,\cdot)$ is a Lie group morphism, when we consider the multiplicative structure of $\R^+$, that is,
\[
c(gh)=c(g)c(h), \qquad \forall g,h\in G.
\]
Indeed,
\[
c(gh) \nu (gh)=\phi _{gh}^*\nu (e)=\phi ^*_h(\phi _{g}^*\nu (e) )=\phi ^*_h(c(g) \nu (g))=c(g)c(h)\nu (gh).
\]
So, $c(gh)=c(g)c(h).$ Note that if  we consider the function $\gamma\colon G\to \R$ characterized by $c=e^\gamma$, then
\[
\gamma (gh)=\gamma (g)+\gamma (h),
\] 
i.e., $\gamma$ is a multiplicative function.
\end{remark}
Using the description of volume forms obtained in Theorem \ref{thm:volume-form} and the previous remark, we can characterize semi-invariant volume forms.
\begin{theorem}\label{2.5}
A volume form $\nu$ on the homogeneous space $M=G/H$ is semi-invariant if and only if the function $\sigma\colon G\to \R$ associated with $\nu$ (deduced from Theorem \ref{thm:volume-form})
is multiplicative.
\end{theorem}
\begin{proof}

If $\phi\colon G\times G/H\to G/H$ is the transitive action of $G$ on $G/H$, $g\in G$, and $L_g\colon G\to G$ is the left-translation by $g\in G$, it follows that $\phi _g\circ q=q\circ L_g$. Therefore, using that $q^*$ is an injective morphism, we deduce that 
\[\phi _g^*\nu = c(g)\nu,\qquad \forall g\in G, \]
if and only if 
\[ c(g) q^*\nu = q^* (\phi _g^*\nu)=L^*_g (q^*\nu),\qquad \forall g\in G, \]
or, equivalently,
\begin{equation}\label{eq:formula}
c(g)(e^\sigma\lvec{\V})=e^{\sigma\circ L_g}\lvec{\V}, \qquad \forall g\in G .
\end{equation}
Therefore, since $\V\neq 0$, \eqref{eq:formula} holds if and only if
\begin{equation}\label{eq:casi-multiplicativo}
    \sigma\circ L_g=\sigma +\log (c(g)), \qquad \forall g\in G .
\end{equation}
But, using that $\sigma (e)=0$, we conclude that \eqref{eq:casi-multiplicativo} is equivalent to 
\[
\sigma (g)=\log (c(g)).
\]
Consequently, $c$ is a Lie group morphism if and only if $\sigma$ is multiplicative. 
\end{proof}
These previous results allow us to characterize whether a homogeneous space admits semi-invariant volume forms. 
\begin{proposition}\label{cor:caract:semi-invariante}
The homogeneous space $M=G/H$ admits a semi-invariant volume form if and only if the 1-cocycle $\chi_\g{}_{|\h}-\chi_\h$ on $\h$ admits a 1-cocycle extension
$\theta_0$ to $\g$ and the closed left-invariant 1-form $\lvec{\theta_0}$ on $G$ is exact. Moreover, in such a case, a semi-invariant volume form $\nu$ on $M$ is characterized by the condition $q^*\nu=e^\sigma\lvec{\V}$, where $\sigma\colon G\to \R$ is the unique multiplicative function on $G$ satisfying $\lvec{\theta_0}=d\sigma$ and $\V\in\wedge^{m-n}\h^0$, $\V\neq 0$ and $d^\g \V=-d\sigma (e)\wedge \V$.
\end{proposition}
\begin{proof}
    If $\nu$ is a semi-invariant volume form then, using Theorems \ref{thm:volume-form} and \ref{2.5}, there is a multiplicative function $\sigma\in C^\infty(G)$ such that 
    \[ q^*\nu = e^\sigma \lvec{\V}, \]
    with $\V\in \wedge^{m-n}\h^0$ and $d\sigma(e)_{|\h}=\chi_\g{}_{|\h}-\chi_\h$. 
    Now, it is enough to take $\theta_0=d\sigma (e)$. In this case, $\theta_0{}_{|\h}=\chi_\g{}_{|\h}-\chi_\h$ and $\lvec{\theta_0}=d\sigma$. Note that the last equality is a consequence of the fact that $\sigma$ is multiplicative.

    Conversely, from the hypothesis, there is a unique  multiplicative function $\sigma\in C^\infty (G)$ such that $\lvec{\theta_0}=d\sigma$ and 
    \[d\sigma(e)_{|\h}=\chi_\g{}_{|\h}-\chi_\h. \]
    Using that $\sigma$ is multiplicative,  we have $\sigma (e)=0$ and $\lvec{X}(\sigma)=X(\sigma)$. Let $\V\in\wedge^{m-n}\h^0$ such that $\V\neq 0$, then (see Lemma \ref{lem:lema0}) there is $\theta_0'\in\g^*$ such that 
    \[ d^\g \V =-\theta_0'\wedge \V\]
    and $\theta_0'{}_{|\h}=\chi_\g{}_{|\h}-\chi_\h$. Since $\theta_0'{}_{|\h}=d\sigma (e)_{|\h}$ then (see Lemma \ref{lem:lema0})
    \[ d^\g \V =-d\sigma (e)\wedge \V.\]
    Now, using Theorem \ref{thm:volume-form}, there is a unique volume form $\nu$ on $G/H$ such that
    \[q^*\nu =e^\sigma \lvec{\V}. \]
    Since $\sigma$ is multiplicative (see Proposition \ref{2.5}) then $\nu$ is semi-invariant.
\end{proof}
\begin{remark}\label{rem-2.7}
    \begin{itemize}
        \item[i)] If $\g=[\g,\g]$ (for instance, if $\g$ is semisimple) there are not non-trivial 1-cocycles and, in particular, $\chi_\g=0$. If, additionally,
        $\h$ is not unimodular then the homogeneous space $M=G/H$  does not admit semi-invariant volume forms. 
        \item[ii)] If $\h$ is unimodular then the homogeneous space $M=G/H$ admits semi-invariant volume forms. In fact, the 1-cocycle $\chi_\g$ is an extension of $\chi_\g{}_{|\h}$ and $\lvec{\chi_\g}=d (\log (\det Ad ))$.
        \item[iii)] Note that if the first de Rham cohomology group of $G$ vanishes (for instance, when the Lie group $G$ is connected and simply-connected), any  1-cocycle $\theta_0$ can be integrated to a multiplicative function $\sigma : G \to \mathbb R$.
    \end{itemize}
    \label{rem:hunim}
\end{remark}
\begin{corollary}\label{cor:formas-invariantes}
Let $M=G/H$ be a homogeneous space. Then, the following conditions are equivalent:
\begin{itemize}
    \item[i)] $M$ admits an invariant volume form.
    \item[ii)] There exists $\V\in\wedge^{m-n}\h^0$, $\V\neq 0$ and $d^\g\V=0$. 
    \item[iii)] The restriction of $\chi_\g$ to $\h$ coincides with $\chi_\h$.
\end{itemize}
Moreover, if \textit{ii)} holds then an invariant volume form $\nu$ on $G/H$ is characterized by the condition $q^*\nu=\lvec{\V}$.
\label{cor:invariantvf}
\end{corollary}
\begin{proof}
Using Theorem \ref{thm:volume-form}, we deduce that the homogeneous space $M$ admits a volume form $\nu$ if and only if there exist $\sigma\in C^\infty(G)$ and $\V\in \wedge^{m-n}\h^0$ which satisfy \textit{i)} and \textit{ii)} in Theorem \ref{thm:volume-form}. In such a case, $q^*\nu =e^\sigma \lvec{\mathcal{V} _0}$.

Now, if $\phi\colon G\times G/H\to G/H$ is the transitive action of $G$ on $G/H$ and $L\colon G\times G\to G$ is the action of $G$ on itself by left-translations, we have that 
\[ q^*(\phi ^*_g\nu) = L^*_g (q^*\nu), \qquad \forall g\in G.\]
So, using that $q^*$ is an injective morphism, it follows that $\nu$ is $G$-invariant if and only if $q^*\nu$ is $G$-invariant.

To prove that \textit{i)} is equivalent to \textit{ii)}, we must show
that $q^*\nu=e^\sigma \lvec{\V}$ is $G$-invariant if and only if $\sigma$ identically vanishes. But, if $g\in G$ then 
\[
L^*_g (e^\sigma \lvec{\V})=e^{(\sigma\circ L_g)}\lvec{\V}.
\]
Therefore, $e^\sigma \lvec{\V}$ is $G$-invariant if and only if $\sigma\circ L_g=\sigma$, for any $g\in G$. Since $\sigma(e)=0$, we conclude that $e^\sigma\lvec{\V}$ is $G$-invariant if and only $\sigma$ identically vanishes.

Finally, let us prove that \textit{ii)} is equivalent to \textit{iii)}. The result directly follows from Lemma \ref{lem:lema0}, since $\theta_0\wedge \V=0$ if and only if $\theta_0\in \h^0$. 
\end{proof}

\begin{example}
    Let $G$ be any connected Lie group and $H=\{e\}$. 
    Since $\h$ is unimodular then, by Remark \ref{rem:hunim} ii), $G/H \simeq G$ always admits semi-invariant volume forms. Indeed, by Corollary \ref{cor:formas-invariantes}, $G/H$ admits invariant volume forms. Note that a volume form $\nu$ on $G/H\cong G$ is invariant with respect to the homogeneous action if and only if $\nu$ is left-invariant on $G$.
    \label{ex:He1}
\end{example}

\begin{example}[{\bf 2-dimensional homogeneous space admitting invariant and semi-invariant volume forms}]
    Consider the quotient $G/H$, where $G$ is the unique connected and simply connected Lie group with Lie algebra $\g$ given by 
    \begin{equation}
    [X_1,X_2]_{\g}=0, \quad [X_1,X_3]_{\g} =X_2, \quad [X_2,X_3]_{\g} =-X_2,
    \end{equation}
    and $\h = \mathrm{Lie} \, H = \langle X_1 \rangle$. Since $\h$ is $1$-dimensional, then it is unimodular ($\chi_\h = 0$) and Remark \ref{rem:hunim} ii) ensures that it admits semi-invariant volume forms. Moreover, we have that $\chi_\g = X^3$, and thus $\chi_\g{}_{|\h} = \chi_\h = 0$. Then, Corollary \ref{cor:invariantvf} ensures that $G/H$ admits invariant volume forms characterized by $q^* \nu = \lambda \lvec{X^2} \wedge \lvec{X^3}$, for any $\lambda \in \mathbb R^+$. Note that in this example $\h$ is not an ideal of $\g$ and, therefore, $G/H$ does not inherit a natural Lie group structure.
    \label{ex:both1}
\end{example}

\begin{example}[{\bf 3-dimensional homogeneous space admitting semi-invariant but not invariant volume forms}]
    Consider the quotient $G/H$, where $G$ is the unique connected and simply connected Lie group of dimension $4$ with Lie algebra $\g$ given by    
    \begin{equation*}
    [X_3,X_4]_{\g} =-X_3 ,
    \end{equation*}
    and the rest of the brackets are zero,
    $\h = \mathrm{Lie} \, H = \langle X_4 \rangle$. Note that $\g \simeq \mathbb R^2 \times r_2$, where $r_2$ is the solvable (and not unimodular) Lie algebra of dimension 2. The Lie algebra $\h$ is  not again an ideal of $\g$ and therefore, $G/H$ does not inherits a natural Lie group structure. Again, since $\h$ is unimodular, we have that $\chi_\h = 0$ and by Remark \ref{rem:hunim} ii) the homogeneous space $G/H$ admits semi-invariant volume forms. However, in this case $\chi_\g = X^4$ and thus $\chi_\g{}_{|\h} = X^4_{|\h} \neq \chi_\h$. By Corollary \ref{cor:invariantvf}, $G/H$ does not admit invariant volume forms. Using Proposition \ref{cor:caract:semi-invariante}, we can give an example of a semi-invariant volume form as follows: take $\theta_0 = X^4$, which is a 1-cocycle extension of $\chi_\g{}_{|\h} - \chi_\h$ and $\V = \lambda X^1 \wedge X^2 \wedge X^3 \in \bigwedge^3 \h^0$. Then $\mathrm d^{\g} \V = - \theta_0 \wedge \V = \lambda X^1 \wedge X^2 \wedge X^3 \wedge X^4 \neq 0$. The semi-invariant volume form is given by $q^*\nu=e^\sigma\lvec{\V}$, where $\sigma\colon G\to \R$ is the unique multiplicative function on $G$ satisfying $\lvec{X^4}=d\sigma$. 
    \label{ex:onlysemi1}    
\end{example}

\section{Poisson-Lie groups and Poisson homogeneous spaces}\label{sec:PLgroup}
%
%
%
%
%
%
%
%
%
%

A Poisson structure $\Pi_G$ on a Lie group $G$ is {\it multiplicative} if and only if the multiplication $G\times G\to G$ on $G$ is a Poisson map, when on the product manifold $G\times G$ we consider the standard product Poisson structure. In these conditions, $\Pi _G$ is a {\it Poisson-Lie structure} on $G$. More generally, a $k$-vector field $P\in \mathfrak{X}^k(G)$ is said to be multiplicative if the equation
\begin{equation}\label{eq:multiplicativo}
P(gh)=(L_g)_* P(h)+(R_h)_* P(g),
\end{equation}
holds for any $g,h\in G$. Since $G$ is connected, a $k$-vector field $P$ is multiplicative if and only if $P(e)=0$ and ${\mathcal L}_{\lvec{X}}P$ is left-invariant, for every $X\in \g.$ Moreover, if 
$P$ and $Q$ are multiplicative multivector fields on a Lie group $G$, so is their Schouten bracket $[P,Q]$ (see \cite{Lu,ILX}).

Given a multiplicative $k$-vector field, we have that $P(e)=0$, $e\in G$ being the identity element. This fact allows to define the map $\delta _P\colon \g \to \wedge ^k \g$
\[
\delta _P (X)=[\lvec{X}, P](e), \qquad \forall X\in \g.
\]
If $P$ is multiplicative then $\delta_P$ is a 1-cocycle with respect to the adjoint representation $ad^{\wedge^k\g}:\g\times \wedge^k\g\to \wedge^k\g$ of $\g$, that is,
\[
\delta_P  [X, Y]=ad^{\wedge^k\g}(X)(\delta_PY)- ad^{\wedge^k\g}(Y)(\delta_PX) \qquad \forall X,Y \in \g
\]
(see, for instance, \cite{Lu}). Note that 
$$[\lvec{X}, P]= \lvec{\delta _P (X)}.$$
\begin{example}\label{ex:cociclos:infinitesimales}

1. Let $(G, \Pi _G)$ be a Poisson-Lie group and   $X^{\Pi _G}_{f}=i_{df}\Pi_G$ be the Hamiltonian vector field of  a function $f:G\to \R$.
If $f$ is multiplicative, then, $X^{\Pi _G}_{f}$ is a multiplicative vector field. In fact, since $\Pi_G(e)=0,$ we have that $X_f^{\Pi_G}(e)=0$ and if $X\in\g$ then 
$${\mathcal L}_{\lvec{X}}X_f^{\Pi_G}=i(d(\lvec{X}(f)))\Pi_G + i(df){\mathcal L}_{\lvec{X}}\Pi_G.$$ 

Now, using that $df$ is left-invariant, we deduce that
$${\mathcal L}_{\lvec{X}}X_f^{\Pi_G}=i(df)(\lvec{\delta_{\Pi_G}X})=i(df(e))(\delta_{\Pi_G}X).$$
In particular, 
$$\delta_{X_f^{\Pi_G}}X=i(\theta)\delta_{\pi_G}X,$$
where $\theta=df(e).$

2. If $r\in \wedge ^k\g$ and $P=\lvec{r}-\rvec{r}$ then $P$ is multiplicative and
$\delta _P(X)=[X,r]$. In the particular case, when $k=2$ and $r$ induces a Poisson structure $\Pi_G$, $(G, \Pi_G)$ is called a coboundary Poisson-Lie group (see \cite{ILX,Lu}).
\end{example}


Given a Poisson-Lie group $(G,\Pi _G)$, the Jacobi identity for the Poisson structure $\Pi _G$ implies that the dual morphism of $\delta _{\Pi_G}$, 
$$\delta _{\Pi _G}^*=[\cdot,\cdot]_{\mathfrak g^*}:\wedge^2{\mathfrak g}^*\to {\mathfrak g}^*$$
is a Lie bracket on ${\mathfrak g}^*$.  In other words, the pair $(({\mathfrak g},[\cdot,\cdot]_\g),({\mathfrak g}^*,[\cdot,\cdot]_ {\mathfrak g^*}))$ is {\it a Lie bialgebra.} Conversely, if $(({\mathfrak g},[\cdot,\cdot]_\g,$ $({\mathfrak g}^*,[\cdot,\cdot]_ {\mathfrak g^*}))$ is a Lie bialgebra and $G$ is the corresponding connected and simply-connected Lie group with Lie algebra $({\mathfrak g},[\cdot,\cdot]_\g),$ then there exists a unique Poisson-Lie structure $\Pi_G$ on $G$ such that the Lie bialgebra associated with
$\Pi_G$ is just 
\[
(({\mathfrak g},[\cdot,\cdot]_\g),({\mathfrak g}^*,[\cdot,\cdot]_ {\mathfrak g^*})).
\]

In addition, there is a Lie algebra structure $[\cdot,\cdot]_{\g\oplus \g^*}$ on $\g\oplus\g^*$ given by 
\begin{equation}\label{eq:Lie:algebra:double}
[X_1+\xi_1,X_2+\xi_2]_{\g\oplus\g^*}=[X_1,X_2]_\g + (ad^{\g^*})^*_{\xi_1} X_2-(ad^{\g^*})^*_{\xi_2} X_1 + [\xi_1,\xi_2]_{\g^*} + (ad^{\g})^*_{X_1}\xi_2-(ad^{\g})^*_{X_2}\xi_1, 
\end{equation}
for $X_i\in \g$ and $\xi_i\in \g^*.$ Here $(ad^{\g^*})^*:\g^*\times \g\to \g$ and $(ad^{\g})^*:\g\times \g^*\to \g^*$ are the coadjoint actions for $(\g,[\cdot,\cdot]_\g)$ and $(\g^*,[\cdot,\cdot]_{\g^*}),$ respectively, i.e. 
\begin{equation}\label{eq:adast}
(ad^{\g^*})^*_\xi X(\xi')=\langle[\xi',\xi]_{\g^*},X\rangle \mbox{ and } (ad^{\g})^*_X \xi(X')=-\langle\xi,[X,X']_\g\rangle
\end{equation}
(for more details, see \cite{Me,Lu, Vaisman1994poissonbook}).
%

\begin{definition} Let $(G,\Pi_G)$ be a Poisson-Lie group and $H$ a closed subgroup of $G$. A $(G, \Pi_G)$-homogeneous  Poisson structure  $\Pi$ is a Poisson bi-vector field on the quotient space $G/H$ such that the action 
\begin{equation}\label{sigma}\phi:(G,\Pi_G)\times (G/H,\Pi)\to (G/H,\Pi)
\;\;\;\; \phi (g',gH)=g'gH, \qquad \forall g,g'\in G,
\end{equation} 
is Poisson. 
\end{definition}
We have the following characterization of a  homogeneous Poisson structure on $G/H$ (see \cite{Dr}).
\begin{proposition}\label{carac-PH}
Let $(G,\Pi_G)$ be a Poisson-Lie group and $H$ a closed subgroup of $G$. A bi-vector field $\Pi$ on $G/H$ is a $(G, \Pi_G)$-homogeneous Poisson structure if and only if these two conditions are satisfied 
\begin{itemize}
\item[i)] $\Pi(gH)=(\phi_g)_* (\Pi (eH))+q_*\Pi_G(g), \quad \forall g\in G$,
\item[ii)]
$\l_{\Pi(eH)}$ is a Lie subalgebra of $(\g\oplus \g^*, [\cdot,\cdot])$, where $\l_{\Pi(eH)}$ is  the Drinfeld Lagrangian subalgebra associated to $(G/H,\Pi )$, i.e. the subspace of $\g\oplus \g^*$ 
\begin{equation}\label{eq:drinfeld:subalgebra}
\l_{\Pi(eH)}=\{X+\xi \, | \,X\in \g, \, \xi\in \h^0, \, i_\xi\Pi(eH)=x+\h\},
\end{equation}
$\h^0$ being the annihilator of the Lie algebra $\h$ of $H$.  
\end{itemize} 
\end{proposition}
Note that, in the previous proposition, $\Pi (eH)\in \wedge ^2T_{eH}(G/H)\cong \wedge^2(\g/\h)$ and $\g/\h\cong (\h^0)^*$.

In this paper, we will deal with a natural family of Poisson homogeneous spaces.
\begin{definition}
Given a Poisson-Lie group $G$ and a closed Lie subgroup $H$, a Poisson homogeneous space $(G/H,\Pi )$ is called a coisotropic Poisson quotient if $H$ is a coisotropic submanifold of  $G$, that is, the subset of functions which vanish on $H$ is a subalgebra of the Poisson algebra of functions on $G$ induced by the Poisson structure $\Pi_G.$ 
\end{definition}
Examples of this type of Poisson homogeneous spaces may be found in Section \ref{sec:exPH}. For this kind of Poisson homogeneous spaces, we have the following characterization \cite{Lu}.
\begin{proposition}\label{cois}
Let $(G, \Pi_G)$ be a Poisson-Lie group and $H$ be a closed subgroup of $G$, such that $(G/H, \Pi)$ is a Poisson homogeneous space.  Then, the following are equivalent:
\begin{itemize}
\item[i)] $(G/H,\Pi )$ is a coisotropic Poisson quotient; 
\item[ii)] $q: (G, \Pi_G) \to (G/H, \Pi)$ is a Poisson epimorphism;
\item[iii)] $\Pi (eH)=0$;
\item[iv)] $\h^0=\{\xi\in \g^* \, | \, \xi_{|\h}=0\}$ is a Lie subalgebra of $(\g^*, [\cdot,\cdot]_{\g^*})$.
\end{itemize}
\end{proposition}
\begin{remark}\label{rmk:3.6}
a) If $G/H$ is a coisotropic Poisson quotient then 
$\l _{\Pi (eH)} =\{X+\xi \, | \, X\in \h, \xi\in \h^0\}\cong \h\oplus \h^0$.

b) Given a Poisson homogeneous space $(G/H,\Pi)$, if we further assume that $\h^0$ is an ideal in $\g^*,$ then $H$ is a Poisson submanifold of $G$. In such a case, $H$ is called a Poisson-Lie subgroup of $(G,\Pi _G)$.
\end{remark}

\section{Unimodularity and multiplicative unimodularity in coisotropic Poisson homogeneous spaces}\label{sec:unimodularity}

In this section, we will study the existence of semi-invariant volume forms on coisotropic Poisson homogeneous spaces which are preserved by all the Hamiltonian vector fields.

Given an orientable $m$-dimensional Poisson manifold $(M,\pi)$, we denote by $Ham _\pi (M)$ the set of Hamiltonian vector fields. If we consider $\nu \in \Omega^m (M)$ a volume form on it, the \textit{modular vector field} associated to $\nu$ is defined as the unique vector field $\M_\nu^\pi \in \mathfrak X (M)$ such that 
\begin{equation}\label{14'}
\M_\nu^\pi (h) = \mathrm{div}_\nu (X^\pi_h)
\end{equation}
for $X^\pi_h = i_{dh}\pi\in Ham_\pi (M)$, $\mathrm{div}_\nu (X^\pi_h)$ being the divergence of  $X^\pi_h$ with respect to $\nu$, \textit{i.e.}
\begin{equation}\label{eq:modular-class}
 \mathrm{div}_\nu (X^\pi_h) \nu = \mathcal L_{X^\pi_h} \nu.
\end{equation}
Note that if we consider other volume form $\nu',$ then there is a function $\sigma:M\to \R$ such that $\nu'=e^\sigma\nu$ and 
\begin{equation}\label{tilde-nu}
\M_{\nu'}^\pi=\M_{\nu}^\pi -\pi^\sharp (d\sigma),
\end{equation}
where $\pi^\#:T^*M\to TM$ is the vector bundle morphism induced by the Poisson $2$-vector $\pi$ (so, $\pi^\#(d\sigma)=X_\sigma^{\pi}$).

In \cite{ELW} (see also \cite{gutierrez2023unimodularity}), it was obtained that the expression of the modular vector field of a Poisson-Lie group $(G,\Pi _G)$ associated with a left invariant volume form $\nu^l$ on $G$ ($\nu\in\wedge^m\g^*,\nu\not=0$ and $m=\dim G)$ is 

\begin{equation}\label{eq:clase:modular:G}
\M_{\nu^l}^{\Pi _G}= \frac12 \left( \lvec{\chi _{\g^*}} + \rvec{\chi_{\g^*}} + \Pi_G^\sharp (\rvec{\chi_\g}) \right).    
\end{equation}
Next, we will describe the modular class of a coisotropic Poisson quotient $(G/H,\Pi)$.
\begin{proposition}\label{prop:modular-class}
Let $(G/H,\Pi)$ be a coisotropic Poisson quotient of the Poisson-Lie group $(G,\Pi_G)$ and $\nu$ be a 
volume form on $G/H$ such that $q^*\nu=e^\sigma\lvec{\V}$, with $\sigma:G\to \R$ a function on $G$, ${\V}\in \wedge^{m-n}\h^0$ and $\V\not=0$ as in Theorem \ref{thm:volume-form}. Then, there is an element $x_{\h^0}$ of $\g$  which projects on $\chi _{\h^0}\in (\h^0)^*\cong \g/\h,$ such that 
  the modular vector field  $\M_\nu^\Pi\in\X(G/H)$ of $(G/H,\Pi)$ is given by 
\begin{equation}\label{eq:proyeccion}
\M_\nu^\Pi=q_*\left(-\Pi _G^\sharp (d\sigma ) +  \frac12 \left(  \rvec{\chi_{\g^*}} - \lvec{\chi _{\g^*}}+2 \lvec{x_{\h ^0}} + \Pi_G^\sharp 
 (\rvec{\chi_\g})\right)\right).
 \end{equation}
Here $q:G\to G/H$ is the quotient projection. 
\end{proposition}
\begin{proof}
Let $\{Y^1,\ldots ,Y^{m-n}\}$ a basis of $\h^0$ such that 
$\V=Y^1\wedge \ldots \wedge Y^{m-n}.$ We complete it to a basis $\{ X^1,\ldots, X^n,Y^1,$ $\ldots ,Y^{m-n}\}$ of $\g^*.$ Denote by $\{ X_1,\ldots, X_n,Y_1,\ldots ,Y_{m-n}\}$ the corresponding dual basis.  In such a case, $\{ X_1,\ldots, X_n\}$ is a basis of $\h$. 

Then,  the 
form $\tilde{\nu}$ given by
\begin{equation*}
\tilde{\nu} = \lvec{X}^1\wedge \ldots \wedge\lvec{X}^n\wedge q^*\nu=e^\sigma \lvec{X}^1\wedge \ldots \wedge\lvec{X}^n\wedge\lvec{\V}
\end{equation*}
is a volume form on $G$.

If $h\in C^\infty (G/H)$, using (\ref{tilde-nu}) and \eqref{eq:clase:modular:G}, we have the relation
\begin{equation}\label{eq:ayuda1}
\M_{\tilde{\nu}} ^{\Pi _G} (h\circ q) {\tilde{\nu}} =\left ( 
-\Pi _G^\sharp (d\sigma ) +
\frac12 ( \lvec{\chi _{\g^*}} + \rvec{\chi_{\g^*}} + \Pi_G^\sharp (\rvec{\chi_\g}) )\right ) (h\circ q)\,{\tilde{\nu}}  .
\end{equation}
On the other hand,
\begin{eqnarray*}
\M_{\tilde{\nu}} ^{\Pi _G} (h\circ q) {\tilde{\nu}} &=& \mathcal L_{X^{\Pi_G}_{h\circ q}} \tilde\nu = \left (\mathcal L_{X^{\Pi_G}_{h\circ q}}  ( \lvec{X}^1\wedge \ldots \lvec{X}^n) \right )\wedge q^*\nu  +  \lvec{X}^1\wedge \ldots \lvec{X}^n \wedge \left (\mathcal L_{X^{\Pi_G}_{h\circ q}}  q^*\nu \right )\\
&=& \sum _i   \lvec{X}^1\wedge \ldots \wedge \left ( \mathcal L_{X^{\Pi_G}_{h\circ q}} \lvec{X}^i \right ) \wedge \ldots  \lvec{X}^n \wedge q^*\nu  +  \lvec{X}^1\wedge \ldots \lvec{X}^n \wedge \left ( (\M ^\Pi _\nu (h) \circ q ) q^*\nu \right ),
\end{eqnarray*}
where, in the last equality, we have used that, since  $G/H$ is a coisotropic Poisson quotient, then $q\colon (G,\Pi _G)\to (G/H,\Pi )$ is a Poisson map and $q_*(X_{h\circ q}^{\Pi_G})=X_h^\Pi$. Moreover, using that $q^* (dh)=\sum _j h_j \lvec{Y}^j$, with $h_j\in C^\infty (G)$,
\begin{eqnarray*}
\sum _i  (\mathcal L_{X^{\Pi_G}_{h\circ q}} \lvec{X}^i ) (\lvec{X}_i)&=&-\sum _i \lvec{X}^i ( [ X^{\Pi_G}_{h\circ q}, \lvec{X}_i ] )=  \sum _i ( \mathcal L_{\lvec{X}_i}\Pi _G ) (q^* dh, \lvec{X}^i) + \sum_i\Pi_G({\mathcal L}_{\lvec{X_i}}(q^*(dh)),\lvec{X^i}) \\ &=& \sum _i \lvec{\delta _{\Pi _G}(X_i) } (q^* dh, \lvec{X}^i) =\sum _{i,j}  h_j  \lvec{\delta _{\Pi _G}(X_i) } (  \lvec{Y}^j , \lvec{X}^i) \\ &=&\sum _{i,j}  h_j   [  Y^j ,  X^i]_{\g^*}(X_i) = \sum _j h_j ( \chi _{\g^*}(Y^j)-  x_{\h^0} (Y^j) ) \\
 &=&  ( \lvec{\chi _{\g^*}} -  \lvec{x_{\h^0}}) (q^*dh), 
\end{eqnarray*}
where $x_{\h^0}\in \g$ is an element which projects on $\chi _{\h ^0}$. Note that in the previous equality, we have used that ${\mathcal L}_{\lvec{X_i}}(q^*(dh))=0$ (this follows from $q^*(dh)$ is a $q$-basic $1$-form and  $\lvec{X_i}$ is a $q$-vertical vector field).

Therefore, 
\begin{equation}\label{eq:ayuda2}
\M_{\tilde\nu} ^{\Pi _G} (h\circ q) {\tilde{\nu}} 
= \left ( \lvec{\chi _{\g^*}} (q^*dh )-  \lvec{x_{\h^0}} (q^*dh)   
+ (\M ^\Pi _\nu (h) \circ q ) \right ) {\tilde{\nu}} .
\end{equation}

Combining equations \eqref{eq:ayuda1} and \eqref{eq:ayuda2}, we conclude that (\ref{eq:proyeccion}) holds. 
\end{proof}
\begin{remark}
In \cite[Lemma 4.13]{Lu08}, it is shown that the modular vector field on $G/H$ for a volume form $\nu$, with $q^*\nu =e^\sigma \lvec{\V}$, is given by \footnote{Note that $\chi_\g$ is $Ad_G$-invariant and thus $\rvec{\chi_\g} =\lvec{\chi_\g}$.}
\begin{equation}\label{eq:clase-modular-PHS-Lu}
\M_\nu^\Pi = q_* \left( -\Pi_G^\sharp (d\sigma) +\frac12 \left( \lvec{x_\l} + \rvec{\chi_{\g^*}} + \Pi_G^\sharp (\rvec{\chi_\g})\right) \right),    
\end{equation}
where 
$x_\l \in \g$ is an element of $\g$ that satisfies that
\begin{equation}\label{eq:x_l}
x_\l (\xi) = \chi_\l (\xi) 
\end{equation}
for all $\xi \in \h^0$.

Now, given a coisotropic Poisson homogeneous quotient $(G/H,\Pi )$, the element 
\[ x_\l = -\chi _{\g^*} + 2 x_{\h^0} \in \g\]
satisfies (\ref{eq:x_l}). Here, $\chi_{\h^0}\in \g$ is an element which projects on $\chi_{\h^0}.$ Indeed, suppose that $\{ X_i \}$ is a basis on $\h$ and complete it to a basis on $\{ X_i, Y_\alpha \}$ on $\g$. We denote the dual basis on $\g^*$ by $\{ X^i, Y^\alpha \}$. Therefore, $\l = \h \oplus \h^0 = \langle X_i \rangle \oplus \langle Y^\alpha \rangle$ and $\l^* = \h^* \oplus (\h^0)^* = \langle X^i \rangle \oplus \langle Y_\alpha \rangle$.

Then, for any $\xi \in \g^*$, we have that
\begin{equation}
 \chi_{\g^*} (\xi) = \Tr \mathrm{ad}^{\g^*}_\xi = (\mathrm{ad}^{\g^*}_\xi X^i) (X_i) + (\mathrm{ad}^{\g^*}_\xi Y^\alpha) (Y_\alpha). 
\end{equation}
Furthermore, for any $\xi \in \h^0$,  using \eqref{eq:Lie:algebra:double} and \eqref{eq:adast}, we deduce
\begin{equation*}
\begin{split}
\chi_\l (\xi) &= \Tr \mathrm{ad}_\xi^{\l} =  (\mathrm{ad}^\l_\xi X_i) (X^i) + (\mathrm{ad}^\l_\xi Y^\alpha) (Y_\alpha) = - (\mathrm{ad}^{\g^*}_\xi X^i) (X_i) + (\mathrm{ad}^{\g^*}_\xi Y^\alpha) (Y_\alpha) \\
&= - \chi_{\g^*} (\xi) + 2 (\mathrm{ad}^{\g^*}_\xi Y^\alpha) (Y_\alpha) = - \chi_{\g^*} (\xi) + 2 x_{\h^0} (\xi) .
\end{split}
\end{equation*}
Thus, \eqref{eq:clase-modular-PHS-Lu} reduces to 
\eqref{eq:proyeccion}.

\end{remark}


\begin{theorem}\label{th:uni-perp}
If $(G/H, \Pi)$ is a unimodular coisotropic Poisson quotient, that is $\M_\nu^\Pi = X_{\hat F}^\Pi$ with $\hat F \in \mathcal C^\infty (G/H)$, then $\h^0$ is a unimodular Lie algebra.
\label{th:h0unim}
\end{theorem}
\begin{proof}
From equation \eqref{eq:proyeccion}, we have that $\M_\nu^\Pi (e H) = q_* (x_{\h^0})$, where we have taken into account that $\Pi_G (e) = 0$. 
Now, by hypothesis, $\M_\nu^\Pi = X^\Pi_{\hat F}$ for some $\hat F \in \mathcal C^\infty (G/H)$. Thus, since $X^\Pi_{\hat F} (eH) = 0$ (from the fact that $\Pi (eH) = 0$), it follows that $q_* (x_{\h^0}) = 0$ and thus $x_{\h^0} \in \h$. 
By definition, $x_{\h^0}$ is an element which projects onto $\chi_{\h^0} \in (\h^0)^*\cong \g/\h,$ 
from which we conclude that $\chi_{\h^0} = 0$.
\end{proof}
If $(G/H, \Pi)$ is unimodular, the results above imply that
\begin{equation}\label{eq:campo:simplificado}
\M_\nu^\Pi = q_* \left( -\Pi_G^\sharp (d\sigma) +\frac12 \left( \rvec{\chi_{\g^*}} -\lvec{\chi_{\g^*}}   + \Pi_G^\sharp (\rvec{\chi_\g})\right) \right).
\end{equation}    
Motivated by this expression, we give the following definition.
\begin{definition}
The vector field $\H ^{\Pi} _\nu$ of $G$ 
\begin{equation}\label{eq:clase-modular-PHS}
 \H ^{\Pi} _\nu =-\Pi_G^\sharp (d\sigma) + \frac12 \left(  \rvec{\chi_{\g^*}} -\lvec{\chi_{\g^*}}  + \Pi_G^\sharp (\rvec{\chi_\g})\right), 
\end{equation}
is called the horizontal modular vector field of $(G/H,\Pi )$ associated to the volume form $\nu.$ Here $\sigma:G\to \R$ is a function  such that $q^*\nu=e^\sigma\lvec\V$ and $\V\in \wedge^{m-n}\h^0,\V\not=0.$
\end{definition}
Note that in the previous definition $\sigma$ is determined  up to an additive constant. So, the vector field $\H ^{\Pi} _\nu$ is the same for all the functions $\sigma.$

Let us show some properties of the horizontal modular class.
\begin{proposition}\label{prop:propiedades:H}
Given a semi-invariant volume form $\nu$ on a coisotropic Poisson quotient $(G/H,\Pi)$, the horizontal modular vector field  $\H ^\Pi _\nu$ is a  $q$-projectable multiplicative and Poisson vector field.
\end{proposition}
\begin{proof}
From \eqref{eq:campo:simplificado},
it follows that $\H_\nu^\Pi$ is $q$-projectable. On the other hand, in \cite{ELW}, 
it also is shown that the vector field 
\begin{equation*}
\frac12 \left( \rvec{\chi_{\g^*}}-\lvec{\chi_{\g^*}}   + \Pi_G^\sharp (\rvec{\chi_{\g}})\right)
\end{equation*}
is multiplicative and Poisson. Thus, since $\sigma$ is multiplicative when $\nu$ is a semi-invariant form (see Theorem \ref{2.5}), we have that the vector field $\Pi_G^{\#}(d\sigma)$ is multiplicative (see Example \ref{ex:cociclos:infinitesimales}) and Poisson. This implies that
\[
\H_\nu ^\Pi =\frac12 \left( \rvec{\chi_{\g^*}}-\lvec{\chi_{\g^*}}   + \Pi_G^\sharp (\rvec{\chi_{\g}})\right)-\Pi_G^\sharp (d\sigma) 
\]
also is a multiplicative and Poisson vector field.    
\end{proof}

The previous proposition and the fact that the Schouten bracket preserves multiplicative multivector fields suggests to consider the sequence
\begin{equation}\label{eq:sucesion:exacta}
\xymatrix{C^\infty _{\mathrm{mult,bas}}(G) \ar[r]^{[\Pi_G, \cdot ]} &
\mathfrak{X}_{\mathrm{mult,proj}}(G) \ar[r]^{[\Pi_G, \cdot ]} & \mathfrak{X}^2_{\mathrm{mult,proj}}(G)\ar[r]&...\ar[r]&\mathfrak{X}^m_{\mathrm{mult,proj}}(G),}
\end{equation}
where $C^\infty _{\mathrm{mult,bas}}(G)$ is the set of $q$-basic functions on $G$ which are multiplicative and 
$\mathfrak{X}^k_{\mathrm{mult,proj}}(G)$ is the set of multiplicative and projectable $k$-vector fields on $G$. Note that the Schouten-Nijenhuis bracket of two projectable multivector fields is again projectable. So, the sequence \eqref{eq:sucesion:exacta} defines a subcomplex of the Poisson cohomology complex of $(G,\Pi_G)$ introduced by Lichnerowicz \cite{Li} and of the projectable Poisson cohomology associated with the Poisson submersion $q:(G,\Pi_G)\to (G/H,\Pi)$ defined in \cite{RC}. 

A 1-cocycle for this subcomplex is a multiplicative and projectable Poisson vector field. So, Proposition \ref{prop:propiedades:H} implies that $\H_\nu ^\Pi$ is 
a 1-cocycle in this cohomology.

This motivates to introduce the following definition.

\begin{definition}\label{definition:multiplicative-unimodular}
Let $(G,\Pi _G)$ be a Poisson-Lie group and $H$ a closed Lie subgroup such that $(G/H, \Pi )$ is a coisotropic Poisson quotient. $(G/H, \Pi )$ is said to be multiplicative unimodular if there exists a semi-invariant volume form $\nu$ and a $q$-basic multiplicative function $\mu \colon G\to\R$ such that
\begin{equation}\label{eq:coborde}
\H ^\Pi _{\nu } = X^{\Pi _G}_{\mu },
\end{equation}
or equivalently  the cohomology class of $\H^\Pi_\nu$ associated with the subcomplex $\mathfrak{X}^k_{\mathrm{mult,proj}}(G)$ of the Poisson cohomology  complex is zero. 
\end{definition}
\begin{proposition}\label{prop:clase0}
    If $(G/H,\Pi )$ is multiplicative unimodular with respect to a semi-invariant volume form $\nu$ and to a $q$-basic multiplicative function $\mu$ then $e^{-\hat{\mu }}\nu$ is a semi-invariant volume form and $\H_{e^{-\hat{\mu }}\nu}^\Pi=0$, with 
$\hat{\mu }\circ q=\mu $.
\end{proposition}
\begin{proof}
We have that
\begin{equation}\label{eq:es hamilt}
\H ^\Pi_\nu =X^{\Pi _G}_{\mu}.
\end{equation}
Since $\nu$ is semi-invariant, then $q^*\nu=e^\sigma\lvec{\V}$, where $\sigma$ is  a multiplicative function and $\V\in\wedge^{m-n}\h^0$, $\V\neq 0$. Taking the volume form $e^{-\hat{\mu}}\nu,$ we will show that it is
semi-invariant. Indeed, for $g\in G,$ the diagram 
\[
\xymatrix{ G \ar[r]^{L_g} \ar[d]_q & G \ar[r]^{\mu}\ar[d]^{q}&{\Bbb R}  \\ G/H \ar[r]_{\phi_g} & G/H\ar[ur]_{\widehat{\mu}}& }
\]
is commutative. Thus, since $\mu$ is multiplicative, we deduce that 
$$(\widehat{\mu}\circ \phi_g)\circ q=\mu\circ L_g=\mu + \mu(g)=(\widehat{\mu} + \mu(g))\circ q$$ 
and therefore, 
$$\widehat{\mu}\circ \phi_g=\widehat{\mu} + \mu(g).$$
This implies that 

\[
    \phi^*_g (e^{-\hat{\mu}}\nu ) = e^{-\hat{\mu}\circ \phi _g}\phi^*_g(\nu ) = e^{-\hat{\mu}-\mu (g) } e^{\sigma (g)} \nu  =  e^{(\sigma - \mu) (g)} \left (e^{-\hat{\mu}} \nu \right ).
\]
So, using that $\mu$ is a multiplicative function, so is $\sigma-\mu$, and we can conclude that $e^{-\hat{\mu}}\nu$ is semi-invariant. 
Note that 
$$q^*\nu=e^{(\sigma-\mu)}\lvec{\V}$$
and, thus
$$\H ^\Pi _{e^{-\hat{\mu}}\nu} =-\Pi_G^\sharp (d(\sigma-\mu)) + \frac12 \left(  \rvec{\chi_{\g^*}} -\lvec{\chi_{\g^*}}  + \Pi_G^\sharp (\rvec{\chi_\g})\right)=\H ^\Pi _\nu -X_\mu^{\Pi_G}=0.$$

Here, we have used  \eqref{eq:proyeccion}, \eqref{eq:clase-modular-PHS} and \eqref{eq:es hamilt}.

Next, we will prove an (almost) infinitesimal characterization of multiplicative unimodular coisotropic Poisson quotient. 
\end{proof}
\begin{theorem}\label{thm:caracterizacion}
Let $(G,\Pi _G)$ be a Poisson-Lie group and $H$ a closed Lie subgroup such that $(G/H, \Pi )$ is a coisotropic Poisson quotient. $(G/H, \Pi )$ is multiplicative unimodular if and only if
\begin{itemize}
    \item[i)] $\h^0$  is unimodular.
    \item[ii)] There exists $\V\in \wedge ^{m-n}\h^0$, $\V\neq 0$ and $\theta_0\in\g^*$ 1-cocycle such that
    \[ d^\g\V=-\theta _0\wedge \V \]
and \[\frac12 \left ([X,\chi_{\g^*}] -i(\chi_{\g})\delta _{\Pi _G}(X) \right ) +i(\theta _0)\delta _{\Pi _G}(X)=0 , \qquad \forall X\in\g.\]
 \item[iii)] The closed left-invariant 1-form $\lvec{\theta_0}$ is exact, that is $\lvec{\theta_0}=d\sigma$, with $\sigma\colon G\to \R$ a multiplicative function.
\end{itemize}
\label{th:semipreservediff}
\end{theorem}
\begin{proof}
Suppose that $(G/H, \Pi)$ is multiplicative unimodular. Then, there is a  semi-invariant volume form $\nu$ on $G/H$. From Proposition \ref{cor:caract:semi-invariante}, $q^*\nu=e^\sigma\lvec{\V}$, where $\sigma\in C^\infty(G)$ is a multiplicative function, 
$\V\in \wedge^{m-n}\h^0$, $\V\neq 0$, and $d^\g \V=-d\sigma (e)\wedge \V$.
Thus, $\theta_0=d\sigma(e)\in \g^*$ is a $1$-cocycle. 

Following Proposition \ref{prop:clase0}, we will can suppose that $\H_\nu ^\Pi=0$ without loss of generality. Then $i)$ follows using Theorem \ref{th:uni-perp}. Moreover, from Proposition \ref{prop:propiedades:H}, $\H_\nu^\Pi$ 
is a multiplicative vector field and, therefore, the corresponding 1-cocycle $\delta_{\H_\nu^\Pi}\colon \g\to \g$ given by
\[
\delta _{\H_\nu^\Pi} (X) = \frac12 \left (- [X,\chi_{\g^*}] +i(\chi_{\g})\delta _{\Pi _G}(X) \right ) -i(\theta _0)\delta _{\Pi _G}(X), \qquad \forall X\in\g,
\]
vanishes, which proves {\it ii)}.


 Finally, since $\sigma$ is a multiplicative function and $\theta_0=d\sigma(e)$, we have that $\lvec{\theta_0}=d\sigma$, which implies $iii)$.

 Conversely, assume that $i)$, $ii)$ and $iii)$ hold. Then,  using $ii)$,  $iii)$ and Proposition \ref{cor:caract:semi-invariante}, there exists a semi-invariant volume form
 on $G/H$ such that $q^\ast \nu =e^\sigma \lvec{\V}$. 
 
On the other hand, if $\H_\nu^\Pi$ is the multiplicative vector field on $G$ given by (\ref{eq:clase-modular-PHS}) then the $1$-cocycle $\delta_{\H_\nu^\Pi}:\g\to\g$ is
 \[\delta _{\H_\nu^\Pi} (X)= \frac12 \left (- [X,\chi_{\g^*}] +i(\chi_{\g})\delta _{\Pi _G}(X) \right )-i(\theta _0)\delta _{\Pi _G}(X), \qquad \forall X\in\g.\]
Thus, using $ii)$, it follows that
\begin{equation}\label{eq:derivada:en:h}
\delta_{\H_\nu^\Pi}(X)=0, \qquad \forall X\in \mathfrak{g}
\end{equation}
and, as a consequence, $\H_\nu^\Pi=0$, i.e., $(G/H,\Pi)$ is multiplicative unimodular.
\end{proof}

\begin{example}
    Consider the Example \ref{ex:He1}, in which $G$ was any connected Lie group and $H=\{e\}$. 
    
    Suppose that $G/H\cong G$ is multiplicative unimodular. Then, using Theorem \ref{thm:caracterizacion} and since $\h^0=\g^*$, we conclude that $\g^*$ is unimodular. Conversely, assume that $\g^*$ is unimodular.  Then, we take $\theta_0=\frac{1}{2} \chi_{\g}\in \g^*.$ We have that $\theta_0$ is a $1$-cocycle. In fact, the closed left-invariant form $\lvec{\theta_0}$ is exact and $\lvec{\theta}_0=\frac{1}{2}\lvec{\chi}_\g=\frac{1}{2}d\sigma,$ with $\sigma:G\to \R$ the modular function on $G$, which is multiplicative. Thus, taking  an arbitrary $\V\in \wedge^m\g^*,$ $\V\not=0$, we deduce that the conditions in Theorem \ref{thm:caracterizacion} hold and $G/H\cong G$ is multiplicative unimodular.

     Now, this is in agreement with the results presented in \cite{gutierrez2023unimodularity} for Poisson-Lie groups, namely that a Poisson-Lie  group $G$ is unimodular iff $\g^*$ is unimodular, since unimodularity and multiplicative unimodularity are equivalent for Poisson-Lie groups (the function for which the modular vector field is a Hamiltonian vector field is multiplicative). 
    \label{ex:He2}
\end{example}

\begin{example}[{\bf 2-dimensional homogeneous space admitting invariant and semi-invariant volumes}]
    Consider the Example \ref{ex:both1}. We recall that $\g$ is given by 
    \begin{equation*}
    [X_1,X_2]_{\g}=0, \quad [X_1,X_3]_{\g} =X_2, \quad [X_2,X_3]_{\g} =-X_2
    \end{equation*}
    and $\h = \mathrm{Lie} \, H = \langle X_1 \rangle$. Take a Lie bialgebra structure given by 
    \begin{equation*}
    \delta_{\Pi _G} (X_1) = 0, \qquad \delta_{\Pi _G} (X_2) = 0, \qquad \delta_{\Pi _G} (X_3) = X_1 \wedge X_2,
    \end{equation*}
    or equivalently by 
    \begin{equation*}
    [X^1,X^2]_{\g^*}=X^3, \quad [X^1,X^3]_{\g^*} =0, \quad [X^2,X^3]_{\g^*} =0 .
    \end{equation*}
    It is easy to check that $\delta_{\Pi _G} \h \subset \h \wedge \h$, or equivalently $[\h^0, \g^*]_{\g^*} \subset \h^0$, and thus $H$ is a Poisson-Lie subgroup. We have that $\chi_\g = X^3$ and $\chi_{\g^*} = 0$, while $\h^0 = \langle X^2, X^3 \rangle$ is unimodular. Taking $\theta_0 = \chi_\g  = X^3$ and $\V = \lambda X^2 \wedge X^3$, we obtain that $d^\g \V = 0$. To prove that this Poisson quotient is multiplicative unimodular, we only have to show that the second part of condition $ii)$ from Theorem \ref{th:semipreservediff} is satisfied. 
    
    Now, since $\chi_{\g^*}=0$ and $\theta_0=\chi_{\g}$, this condition holds.      
    \label{ex:both2}
\end{example}

\begin{example}[{\bf 3-dimensional homogeneous space admitting semi-invariant but not invariant volume forms}]
    Recall Example \ref{ex:onlysemi1}, where the Lie algebra $\g$ was given by    
    \begin{equation*}
    [X_3,X_4]_{\g} =-X_3, 
    \end{equation*}
     the rest of the brackets are zero and $\h = \mathrm{Lie} \, H = \langle X_4 \rangle$. Taking the Lie bialgebra structure given by 
\begin{equation*}
\delta_{\Pi _G} (X_1) = X_1 \wedge X_2, \qquad \delta_{\Pi _G} (X_2) = 0, \qquad \delta_{\Pi _G} (X_3) = X_2 \wedge X_3, \qquad \delta_{\Pi _G} (X_4) = 0,
\end{equation*}
or equivalently by 
\begin{equation*}
[X^1,X^2]_{\g^*}=X^1, 
\; [X^2,X^3]_{\g^*} =X^3
\end{equation*}
and the rest of the brackets are zero. 
Note that $\g^* \simeq r_3(-1) \times \mathbb R$, where $r_3(-1)$ is the Bianchi $\mathrm{VI}_0$ Lie algebra (solvable and unimodular), which is isomorphic to $\mathfrak p (1+1)$. We have that $\delta_{\Pi _G} \h \subset \h \wedge \h$, or equivalently $[\h^0, \g^*]_{\g^*} \subset \h^0$, and thus $H$ is a Poisson-Lie subgroup. Condition $i)$ from Theorem \ref{th:semipreservediff} is satisfied since $\h^0 = \langle X^1, X^2, X^3 \rangle$ is unimodular. Moreover, condition $ii)$ is also satisfied taking $\theta_0 = \chi_\g = X^4$, since if $\V = \lambda X^1 \wedge X^2 \wedge X^3 \in \bigwedge^3 \h^0$ then $\mathrm d^{\g} \V = - \chi_\g \wedge \V \neq 0.$ 
    \label{ex:onlysemi2}    
\end{example}






\section{Unimodularity and semisimple Lie group quotients of dimension two}\label{sec:examples}
\label{sec:exPH}

Let us now illustrate our results for Poisson homogeneous spaces corresponding to Poisson-Lie groups actions of semisimple groups of dimension 3. In general, our Theorems \ref{th:uni-perp} and \ref{th:semipreservediff} are sufficient to prove that neither of the Poisson quotients considered here are multiplicative unimodular. Furthermore, we indeed prove that no Poisson quotient for these Lie groups are unimodular. In same cases (when the annihilator of the isotropy algebra is not a unimodular Lie algebra) this last result may be deduced directly from Theorem \ref{th:uni-perp}. However, in other cases, the proof involves global computations at the Lie group level.

\subsection{Poisson homogeneous spaces on $S^2$}
Consider the Lie group $G=\mathrm{SO}(3)$, with Lie algebra $\mathfrak g = \mathfrak{so}(3)$ given by
\begin{equation*}
[J_1, J_2]_\mathfrak{g}=J_3, \qquad [J_2,J_3]_\mathfrak{g}=J_1, \qquad  [J_3,J_1]_\mathfrak{g}=J_2 .
\end{equation*}
The group $\mathrm{SO}(3)$ is unimodular, and then the modular character of $\mathfrak{so}(3)$ is $\chi_{\mathfrak{so}(3)} = 0$ (see Remark \ref{rem-2.7}). All Poisson-Lie structures on $\mathrm{SO}(3)$ are coboundary, \textit{i.e.} they are defined by a $r$-matrix solution of the modified classical Yang-Baxter equation. In this particular case, there is only one infinite family of $r$-matrices given by
\begin{equation}
r = \p J_1 \wedge J_2 ,
\label{eq:rso3}
\end{equation}
where $\p \in \mathbb{R}^*$ is an essential parameter, \textit{i.e.} if $\p \neq \p'$ then the Poisson-Lie structures are not isomorphic. The associated one-parameter family of Lie bialgebras are defined by the following cocommutator map
\begin{equation*}
\delta_\eta (J_1) = \p J_1 \wedge J_3, \qquad \delta_\eta (J_2) = \p J_2 \wedge J_3, \qquad \delta_\eta (J_3) = 0 ,
\label{eq:cocomso3}
\end{equation*}
or equivalently, the dual Lie algebra $\mathfrak g^*$ is 
\begin{equation*}
[J^1, J^2]_{\mathfrak{g}^*}=0, \qquad [J^1, J^3]_{\mathfrak{g}^*}=\p J^1, \qquad [J^2,J^3]_{\mathfrak{g}^*}= \p J^2 .
\end{equation*}
Therefore, the modular character $\chi_{\mathfrak{g}^*}$ of $\mathfrak g^*$ is $\chi_{\mathfrak{g}^*} = - 2 \eta J_3$.

In order to compute the associated Poisson-Lie structure for $SO(3)$, we instead work with its double cover $SU(2)$. We can parametrize the Lie group $SU(2)$ as
\begin{equation}
    SU(2)= \left\{ 
    \begin{pmatrix}
        x +i y & -z + i t \\
        z +i t & x-i y \\
    \end{pmatrix}
   | \, x,y,z,t \in \mathbb R, \, x^2 + y^2 + z^2 + t^2 = 1 \right\} .
   \label{eq:SU2coords}
\end{equation}
Using this coordinates, a straightforward computation shows that the coboundary Poisson-Lie structure defined by \eqref{eq:rso3} is given by 
\begin{equation*}
\begin{array}{lll}
       \{x,y\} = \frac{1}{2} \p (z^2 + t^2), & \{x,z\} = - \frac{1}{2} \p y z, & \{x,t\} = - \frac{1}{2} \p y t, \\ [10pt]
       \{y,z\} = \frac{1}{2} \p x z, & \{y,t\} = \frac{1}{2} \p x t, & \{z,t\} = 0.  
\end{array}    
\end{equation*}
We see that this Poisson structure is real, and it indeed defines a Poisson-Lie structure on $SO(3)$.

Consider the two different $SO(3)$-covariant Poisson homogeneous structures on $S^2 = SO(3)/S^1$ that are described as Poisson quotients in Table \ref{table:so3PHS}. For the first of these structures, which we call {\it subgroup sphere}, the isotropy subgroup $H$ is a Poisson-Lie subgroup (or equivalently, $\h^0$ is an ideal of $\g^*$), while for the second one (the {\it coisotropic sphere}) $\h^0$ is only a Lie subalgebra of $\g^*$. The Poisson cohomology of these $SO(3)$-covariant Poisson structures on the sphere has been already computed in \cite{Ginzburg,Roytenberg2002}. In particular, it was shown that neither of them are unimodular. Here we present an alternative proof of this result.

\begin{table}[H]
\begin{center}
\renewcommand\arraystretch{1.5}
\begin{tabular}{|c||c|c|c|}
 \hline
 & $r = \eta J_1 \wedge J_2$  \\
    \multirow{3}*        & $\chi_{\g^*} = -2 \p J_3$          \\
 & $[r,r]_\g \neq 0$  \\
\hline
\hline
 \multirow{3}*[1.5em]{$\boxed{G/H=S^2 = \textit{Subgroup sphere}}$}          &  $\chi_{\h^0} = 0$, but not M.U. by Th. \ref{th:semipreservediff} $ii)$ \\
 $\h = \langle J_3 \rangle, \quad \h^0 = \langle J^1,J^2 \rangle$ &  $[\h^0, \g^*]_{\g^*} \subset \h^0 \Rightarrow H$ is a PL subgroup  \\
 \hline
 \multirow{3}*[1.5em]{$\boxed{G/H=S^2 = \textit{Coisotropic sphere}}$}    &      $\chi_{\h^0} \neq 0 \Rightarrow $ Not U. by Th. \ref{th:uni-perp}  \\
 $\h = \langle J_1 \rangle, \quad \h^0 = \langle J^2,J^3 \rangle$  &  $[\h^0, \g^*]_{\g^*} \not \subset \h^0 \Rightarrow H$ is a coisotropic subgroup \\
 \hline
\end{tabular}
\caption{\label{table:so3PHS}
{\footnotesize
The subgroup sphere and coisotropic sphere $SO(3)$-covariant Poisson homogeneous spaces constructed as quotients of $SO(3)$ by different uniparametric subgroups. In the first column, for each quotient we show its isotropy subalgebra $\h$ and its annihilator $\h^0$. In the second column we show the most relevant features of its Poisson homogeneous structure: the modular character $\chi_{\h^0}$ stating if the structure is known not to be unimodular (U.) by Theorem \ref{th:uni-perp} or not be multiplicative unimodular 
 (M.U.) by Theorem \ref{thm:caracterizacion}. In fact, in the case when $H$ is only a coisotropic subgroup ({\it the coisotropic sphere}), the Lie subalgebra $\h^0$ is not unimodular and, thus, the Poisson homogeneous space is not unimodular by Theorem \ref{th:uni-perp}. However, when $\h^0$ is not unimodular, that is, $H$ is a Poisson-Lie (PL) subgroup, we need to use Theorem \ref{thm:caracterizacion} in order to prove that the Poisson homogeneous space is not multiplicative unimodular.
 }
}
    \end{center}    
    \end{table}


In what follows, we will show that no Poisson quotient of $SO(3)$ can be unimodular. 

In the coordinates given in \eqref{eq:SU2coords} we have that
\begin{equation}\label{lxg*}
\begin{array}{rcl}
    \lvec{\chi_{\g^*}} &=& \p \left( y \displaystyle\frac{\partial}{\partial x} - x \displaystyle\frac{\partial}{\partial y} + t \displaystyle\frac{\partial}{\partial z} - z \displaystyle\frac{\partial}{\partial t} \right) ,\\[8pt]
    \rvec{\chi_{\g^*}} &=& \p \left( y \displaystyle\frac{\partial}{\partial x} - x \displaystyle\frac{\partial}{\partial y} - t \displaystyle\frac{\partial}{\partial z} + z \displaystyle\frac{\partial}{\partial t} \right). 
\end{array}
\end{equation}

On the other hand, the modular characteres $\chi_{{\mathfrak so}(3)}$ and  $\chi_{\h}$ are zero  in both cases of the subgroup sphere and  of the coisotropic sphere. From Corollary \ref{cor:formas-invariantes}, we have a invariant volume form $\nu.$

Therefore, 
the vector field defined in \eqref{eq:clase-modular-PHS} is given for this particular case by 
\begin{equation*}
     \H_{\nu}^\Pi=\frac{1}{2} (\rvec{\chi_{\g^*}} - \lvec{\chi_{\g^*}}) = \p \left( -t \frac{\partial}{\partial z} + z \frac{\partial}{\partial t} \right) . 
\end{equation*}

For $(SO(3)/S^1,\Pi)$ to be unimodular,  the modular vector field $\M^\Pi_\nu\in \X(G/H),$ deduced from   $\nu$,   must be a Hamiltonian vector field. We have that $\M_\nu^\Pi$ is Hamiltonian if and only if $\H_\nu^\Pi$ is $q$-basic Hamiltonian, \textit{i.e.} $\H_\nu^\Pi = \Pi^\sharp _G(d \mu) $ for some $q$-basic function  $\mu \in \mathcal C^\infty (G)$. We have that
\begin{equation*}
    \begin{split}
        \Pi^\sharp _G(d \mu) &= \frac{\eta}{2} \bigg( 
        (z^2+t^2) \left( \frac{\partial \mu}{\partial x} \frac{\partial}{\partial y} - \frac{\partial \mu}{\partial y} \frac{\partial}{\partial x} \right) 
        - y z \left( \frac{\partial \mu}{\partial x} \frac{\partial}{\partial z} - \frac{\partial \mu}{\partial z} \frac{\partial}{\partial x} \right) 
        - y t \left( \frac{\partial \mu}{\partial x} \frac{\partial}{\partial t} - \frac{\partial \mu}{\partial t} \frac{\partial}{\partial x} \right) \\
        &+ x z \left( \frac{\partial \mu}{\partial y} \frac{\partial}{\partial z} - \frac{\partial \mu}{\partial z} \frac{\partial}{\partial y} \right)
        + x t \left( \frac{\partial \mu}{\partial y} \frac{\partial}{\partial t} - \frac{\partial \mu}{\partial t} \frac{\partial}{\partial y} \right)
        \bigg) \\
        &= \frac{\eta}{2} \bigg( -(z^2+t^2) \frac{\partial \mu}{\partial y} + y z \frac{\partial \mu}{\partial z} + y t \frac{\partial \mu}{\partial t}\bigg) \frac{\partial}{\partial x} +\frac{\eta}{2} \bigg( (z^2+t^2) \frac{\partial \mu}{\partial x} - x z \frac{\partial \mu}{\partial z} - x t \frac{\partial \mu}{\partial t}\bigg) \frac{\partial}{\partial y} \\
        &+ \frac{\eta}{2} z \bigg( - y \frac{\partial \mu}{\partial x} + x \frac{\partial \mu}{\partial y} \bigg) \frac{\partial}{\partial z}
        + \frac{\eta}{2} t \bigg( - y \frac{\partial \mu}{\partial x} + x \frac{\partial \mu}{\partial y} \bigg) \frac{\partial}{\partial t}.
    \end{split}
\end{equation*}
In particular, it must be 
\begin{equation*}
    z \, \Pi^\sharp _G(d \mu)  (t) - t \, \Pi^\sharp _G (d\mu)(z) = z \W (t) - t \W (z) .
\end{equation*}
However, we have that 
\begin{equation*}
    z \, \Pi^\sharp _G(d \mu)  (t) - t \, \Pi^\sharp _G(d \mu)  (z) = 0 \neq \eta (z^2 + t^2) = z \W (t) - t \W (z).
\end{equation*}
Therefore, we have showed that $\H_\nu^\Pi$ is not Hamiltonian for  the Poisson-Lie structure $\Pi_G$  on $SO(3)$, and in particular, it is not a Hamiltonian vector field for any $q$-basic function. Therefore, no Poisson quotient of $SO(3)/S^1$ can be unimodular. 

\subsection{Poisson homogeneous spaces for $\mathrm{SL}(2,\mathbb R)$}

Let us now study the semisimple group $G = \mathrm{SL}(2,\mathbb R)$, the matrix Lie group defined by $2 \times 2$ matrices with determinant equal to one, \textit{i.e.}
\begin{equation}
    G = SL(2, \mathbb R) = \left\{ 
    \begin{pmatrix}
        x & y \\
        z & t \\
    \end{pmatrix}
    \in GL(2, \mathbb R) \ | \, xt-yz = 1 \right\} .
    \label{eq:SL2g}
\end{equation}
Three interesting quotients for $\mathrm{SL}(2,\mathbb R)$, relevant for $(1+1)$ dimensional gravity (see \cite{BMN2017homogeneous}), are the Anti-de Sitter space (one-sheeted hyperboloid), the two-sheeted hyperbolic space and the light cone. These three spaces are described as quotients in the first column of Table \ref{table:sl2PHS}, in which we have used two different basis for the Lie algebra $\mathfrak{sl}(2,\mathbb R)$. The first one $\{ P_1,P_2,J_{12} \}$, used to describe the Anti-de Sitter and hyperbolic spaces, reads
\begin{equation}
    [P_1, J_{12}]_\mathfrak{g}=-P_2, \qquad [P_2,J_{12}]_\mathfrak{g}=-P_1, \qquad  [P_1,P_2]_\mathfrak{g}=J_{12}
    \label{eq:sl2basis1}
\end{equation}
while the second one $\{ J_+,J_-,J_{3} \}$, which is used to construct the light cone, is given by
\begin{equation}
    [J_3, J_+]_\mathfrak{g}=2 J_+, \qquad [J_3,J_-]_\mathfrak{g}=-2 J_-, \qquad  [J_+,J_-]_\mathfrak{g}=J_3 .
        \label{eq:sl2basis2}
\end{equation}

\begin{table}[H]
{\tiny
\begin{center}
\renewcommand\arraystretch{1.5}
\begin{tabular}{|c||c|c|c|}
 \hline
 \multirow{2}*{}     & Hyperbolic   & Elliptic   & Parabolic   \\
 & $r = 2 \p P_1 \wedge P_2 = \eta J_+ \wedge J_-$ &  $r= 2 \eta J_{12} \wedge P_2 = \eta J_3 \wedge (J_+ + J_-)$ &  $r = \p J \wedge (P_1 + P_2) = \frac12 \p J_3 \wedge J_+$  \\
   \multirow{3}*        & $\chi_{\g^*} = -4 \p J_{12} = - 2 \eta J_3$    & $\chi_{\g^*} = - 4 \p P_1 = -2 \eta (J_+ - J_-)$    & $\chi_{\g^*} = - 2 \p (P_1 + P_2) = - 2 \p J_+$       \\
 & $[r,r]_\g \neq 0$ & $[r,r]_\g \neq 0$ & $[r,r]_\g = 0$ \\
\hline
\hline
 \multirow{3}*[1.5em]{$\boxed{G/H=AdS_2}$}          &    &    &   \\
$\h = \langle J_{12} \rangle$  & $\chi_{\h^0} = 0$, but not M.U. by Th. \ref{th:semipreservediff} $ii)$ & $\chi_{\h^0} \neq 0 \Rightarrow$  Not U. by Th. \ref{th:uni-perp}  & $\chi_{\h^0} \neq 0 \Rightarrow$  Not U. by Th. \ref{th:uni-perp} \\
 $\h^0 = \langle P^1,P^2 \rangle$ & $[\h^0, \g^*]_{\g^*} \subset \h^0 \Rightarrow H$ is a PL subg. & $[\h^0, \g^*]_{\g^*} \not \subset \h^0 \Rightarrow H$ is a cois. subg. &  $[\h^0, \g^*]_{\g^*} \not \subset \h^0 \Rightarrow H$ is a cois. subg. \\
 \hline
 \multirow{3}*[1.5em]{$\boxed{G/H=H_2 \times \mathbb Z_2}$}          &   &   &      \\
$\h = \langle P_1 \rangle$  & $\chi_{\h^0} \neq 0 \Rightarrow$ Not U. by Th. \ref{th:uni-perp} & $\chi_{\h^0} = 0$, but not M.U. by Th. \ref{th:semipreservediff} $ii)$ & $\chi_{\h^0} \neq 0 \Rightarrow$ Not U. by Th. \ref{th:uni-perp} \\
 $\h^0 = \langle P^2,J^{12} \rangle$ & $[\h^0, \g^*]_{\g^*} \not \subset \h^0 \Rightarrow H$ is a cois. subg. & $[\h^0, \g^*]_{\g^*} \subset \h^0 \Rightarrow$ $H$ is a PL subg. &  $[\h^0, \g^*]_{\g^*} \not \subset \h^0 \Rightarrow H$ is a cois. subg. \\
 \hline
 \multirow{3}*[1.5em]{$\boxed{G/H=L_2}$}          &    &    &     \\
 $\h = \langle J_+ \rangle$  & $\chi_{\h^0} \neq 0 \Rightarrow$ Not U. by Th. \ref{th:uni-perp} & $\chi_{\h^0} \neq 0 \Rightarrow$ Not U. by Th. \ref{th:uni-perp} & $\chi_{\h^0} = 0$, but not M.U. by Th. \ref{th:semipreservediff} $ii)$ \\
 $\h^0 = \langle J^3,J^- \rangle$ & $[\h^0, \g^*]_{\g^*} \not \subset \h^0 \Rightarrow H$ is a cois. subg. & $[\h^0, \g^*]_{\g^*} \not \subset \h^0 \Rightarrow H$ is a cois. subg. & $[\h^0, \g^*]_{\g^*} \subset \h^0 \Rightarrow$ $H$ is a PL subg. \\
 \hline
\end{tabular}
\caption{\label{table:sl2PHS}
The Anti-de Sitter space or one-sheeted hyperboloid $AdS_2$, the two-sheeted hyperboloid $H_2 \times \mathbb Z_2$ and the light-cone $L_2$ $SL(2,\mathbb R)$-covariant Poisson homogeneous spaces constructed as quotients of $SL(2,\mathbb R)$ by different uniparametric subgroups. In the first column, for each quotient we show its isotropy subalgebra $\h$ and its annihilator $\h^0$. In the other columns we show, for each of the Poisson-Lie structures on $SL(2,\mathbb R)$, the most relevant features of its Poisson homogeneous structure: the modular character $\chi_{\h^0}$ stating if the structure is known not to be unimodular (U.) by Theorem  \ref{th:uni-perp} or not be multiplicative unimodular (M.U.) by Theorem \ref{thm:caracterizacion}. This depends from the nature of $\h^0$ as a Lie ideal or Lie subalgebra of $\g^*$, which is equivalent to say that the isotropy subgroup $H$ is a Poisson-Lie (PL) subgroup or a coisotropic subgroup.
}
    \end{center}    
    }
    \end{table}

In the following, we explicitly describe the Poisson-Lie structures on $\mathrm{SL}(2,\mathbb R)$ given in Table \ref{table:sl2PHS} and show that neither of the Poisson homogeneous spaces are unimodular. We recall that all Poisson-Lie structures on semisimple Lie groups are coboundary and therefore can be straightforwardly computed using the classical $r$-matrix. The names hyperbolic, elliptic and parabolic for each this $r$-matrices are the ones given in \cite{Reyman1996}.

\subsubsection{Hyperbolic Poisson-Lie structure}

The hyperbolic (also called standard or Drinfel'd-Jimbo) Poisson-Lie structure (first column Table \ref{table:sl2PHS}) is defined by the $r$-matrix
\begin{equation}
r = 2 \p P_1 \wedge P_2 = \p J_+ \wedge J_- ,
\label{eq:rhyper}
\end{equation}
where $\p \in \mathbb{R}-\{0\}$. Its Lie bialgebra structure is given by
\begin{equation}
\delta (J_{12}) = 0, \qquad \delta (P_1) = 2 \p P_1 \wedge J, \qquad \delta (P_2) = 2 \p P_2 \wedge J ,
\end{equation}
or equivalently, by
\begin{equation}
[P^1, J^{12}]_{\mathfrak{g}^*}=2 \p P^1, \qquad [P^2,J^{12}]_{\mathfrak{g}^*}=2 \p P^2, \qquad [P^1,P^2]_{\mathfrak{g}^*}=0 ,
\label{eq:dualhyper1}
\end{equation}
in the basis \eqref{eq:sl2basis1}, or by
\begin{equation}
\delta (J_3) = 0, \qquad \delta (J_\pm) = \p J_\pm \wedge J_3 ,
\end{equation}
or equivalently 
\begin{equation}
[J^3, J^\pm]_{\mathfrak{g}^*}= - \p J^\pm, \qquad [J^+,J^-]_{\mathfrak{g}^*}=0 ,
\label{eq:dualhyper2}
\end{equation}
in the basis \eqref{eq:sl2basis2}. From the explicit expressions of the dual Lie algebra \eqref{eq:dualhyper1} and \eqref{eq:dualhyper2}, using Theorem \ref{th:uni-perp}, we can see that the Poisson homogeneous structures on $H_2 \times \mathbb Z_2$ and $L_2$ induced by this Poisson-Lie structure are not  unimodular since $\chi_{\h^0} \neq 0$, where $\h = \langle J_{12} \rangle$. However, this is not enough for the first Poisson quotient, the 2-dimensional anti-de Sitter space $AdS_2$ (note that this is the only one that is a Poisson quotient of Poisson subgroup type). In order to prove that this Poisson quotient is not multiplicative unimodular, we need to check that the infinitesimal characterization given in Theorem \ref{th:semipreservediff} $ii)$ is not fulfilled, which can be easily seen since $d^\g (\lambda P^1 \wedge P^2) = 0$ and $\theta_0 = 0$ is the only admissible cocycle on $SL(2,\mathbb R)$. 

\paragraph{Not unimodular Poisson homogeneous spaces for the hyperbolic Poisson-Lie structure.} 

As we have just showed, Theorems \ref{th:uni-perp} and \ref{th:semipreservediff} are sufficient to prove that none of the Poisson quotients for the hyperbolic Poisson-Lie structure described in Table \ref{table:sl2PHS} are multiplicative unimodular. In order to complete the analysis of these examples, let us show that no Poisson homogeneous space for the hyperbolic Poisson-Lie structure is unimodular. In order to do that, 
similarly to the $SO(3)$ case, the modular characters $\chi_{{\mathfrak sl}(2,\R)}$  and $\chi_{\mathfrak h}$ are  zero and, from Corollary \ref{cor:formas-invariantes}, we have an invariant volume form $\nu$ on $G/H$.

Suppose that   the modular vector field $\M^\Pi_\nu\in \X(G/H),$ deduced from   $\nu$,   must be a Hamiltonian vector field. We have that $\M_\nu^\Pi$ is Hamiltonian if and only if $\H_\nu^\Pi$ is $q$-basic Hamiltonian, \textit{i.e.} $\H_\nu^\Pi = \Pi^\sharp _G(d \mu) $ for some $q$-basic function  $\mu \in \mathcal C^\infty (G)$. 

Using the parametrization \eqref{eq:SL2g}, the Poisson-Lie structure defined on $SL(2,\mathbb R)$ by \eqref{eq:rhyper} reads
\begin{equation}
\begin{array}{llll}
    &\{x,y\} = \p x y,  &\{x,z\} = \p x z,  &\{x,t\} = 2 \p y z, \\ [10pt]
    &\{y,z\} = 0,  &\{y,t\} = \p y t,  &\{z,t\} = \p z t. 
\end{array}
    \label{eq:poishyper}
\end{equation}
In addition, the modular character of the dual Lie algebra can be computed from \eqref{eq:dualhyper1} and results in
\begin{equation}
    \chi_{\g^*} = -4 \p J_{12} .
\end{equation}
The left- and right-vector fields associated to the modular character read
\begin{equation}
    \lvec{\chi_{\g^*}} = 2 \p \left( -x \frac{\partial}{\partial x} + y \frac{\partial}{\partial y} - z \frac{\partial}{\partial z} + t \frac{\partial}{\partial t} \right), \qquad 
    \rvec{\chi_{\g^*}} = 2 \p \left( -x \frac{\partial}{\partial x} - y \frac{\partial}{\partial y} + z \frac{\partial}{\partial z} + t \frac{\partial}{\partial t} \right)
\end{equation}
and finally the multiplicative vector field $\W$ reads 
\begin{equation}
    \H_\nu^\Pi = 2 \p \left( - y \frac{\partial}{\partial y} + z \frac{\partial}{\partial z} \right) .
    \label{eq:aaaa1}
\end{equation}
In order for $\W$ to be Hamiltonian, it has to satisfy that
\begin{equation}
    \Pi^\sharp _G(d \mu)  (f) = \H_\nu^\Pi (f),
    \label{eq:aaaaa}
\end{equation}
for some $\mu \in \mathcal{C}^\infty (G)$ and all $f \in \mathcal{C}^\infty (G)$. In particular, using \eqref{eq:poishyper}, \eqref{eq:aaaa1} and \eqref{eq:aaaaa} with $f=y$, we have that 
\begin{equation}
    \varphi(x,y,z,t) = - 2 \p ,
\end{equation} 
where we have written
\begin{equation}
    \varphi(x,y,z,t) := x \frac{\partial \mu}{\partial x} - t \frac{\partial \mu}{\partial t} .
\end{equation}
However, from $\Pi^\sharp _G(d \mu)  (z) = \W (z)$ we have that 
\begin{equation}
    \varphi(x,y,z,t) = 2 \p .
\end{equation} 
Therefore, these two equations cannot be simultaneously satisfied and thus $\M_\nu^\Pi$ is not a Hamiltonian vector field. This proves that no Poisson quotient for the hyperbolic Poisson-Lie structure (first column from Table \ref{table:sl2PHS}) on $SL(2,\mathbb R)$ is unimodular.

\subsubsection{Elliptic Poisson-Lie structure}

We now consider the elliptic Poisson-Lie structure on $SL(2,\mathbb R)$ defined by the classical $r$-matrix
\begin{equation}
r = 2 \p J_{12} \wedge P_2 = \p J_3 \wedge( J_+ + J_-) ,
\label{eq:rell}
\end{equation}
where $\p \in \mathbb{R}-\{0\}$. Using the basis \eqref{eq:sl2basis1} the Lie bialgebra structure on $\mathfrak{sl}(2,\mathbb R)$ reads
\begin{equation}
\delta (J_{12}) = 2 \p J_{12} \wedge P_1, \qquad \delta (P_1) = 0, \qquad \delta (P_2) = -2 \p P_1 \wedge P_2 ,
\end{equation}
or equivalently its dual Lie algebra is given by
\begin{equation}
[P^1,P^2]_{\mathfrak{g}^*}=-2 \p P^2, \qquad [P^1, J^{12}]_{\mathfrak{g}^*}=-2 \p J^{12}, \qquad [P^2,J^{12}]_{\mathfrak{g}^*}=0 .
\label{eq:dualell1}
\end{equation}
In terms of the basis \eqref{eq:sl2basis2}, the Lie bialgebra cocycle takes the form
\begin{equation}
\delta (J_3) = 2 \p J_3 \wedge (J_+ - J_-), \qquad \delta (J_+) = -2 \p J_+ \wedge J_- , \qquad \delta (J_-) = -2 \p J_+ \wedge J_- ,
\end{equation}
or equivalently in terms of the dual Lie algebra
\begin{equation}
[J^3, J^+]_{\mathfrak{g}^*}= 2 \p J^3, \qquad [J^3, J^-]_{\mathfrak{g}^*}= - 2 \p J^3 ,\qquad [J^+,J^-]_{\mathfrak{g}^*}= -2 \p (J^+ + J^-) .
\label{eq:dualell2}
\end{equation}
Similarly to the previous examples, Theorem \ref{th:uni-perp} directly guarantees that two of these Poisson quotients, in this case the Anti-de Sitter space (one-sheeted hyperboloid) and the light-cone, are not unimodular. In fact,   from \eqref{eq:dualhyper1} and \eqref{eq:dualhyper2}, we can see that in these two cases $\chi_{\h^0} \neq 0$. 

It remains to prove that the two-sheeted hyperboloid $H_2 \times \mathbb Z_2$ is not multiplicative unimodular (remarkably, this is again the only Poisson quotient of Poisson subgroup type for this structure). In order to do so, we again check that the infinitesimal characterization given in Theorem \ref{th:semipreservediff} $ii)$ is not satisfied, since $d^\g (\lambda P^2 \wedge J^{12}) = 0$ and $\theta_0 = 0$ is the only admissible cocycle on $SL(2,\mathbb R)$.

\paragraph{Not unimodular Poisson homogeneous spaces for the elliptic Poisson-Lie structure.} 

Similarly to the previous example, we finally show that there are no unimodular Poisson homogeneous spaces for the elliptic Poisson-Lie structure. We have that $\chi_{{\mathfrak sl}(2,\R)}=0$ and $\chi_{{\mathfrak h}}=0$ and then  there is an invariant volume form $\nu$ on $G/H$ (see Corollary \ref{cor:formas-invariantes}).

The fundamental Poisson brackets for the elliptic Poisson-Lie structure defined by \eqref{eq:rell} are given by 
\begin{equation}
\begin{split}
     \{x,y\} &= \frac{\p}{2} \left( x(t-x) -y(y+z) \right), \; \{x,z\} = \frac{\p}{2} \left( x(x-t) +z(y+z) \right), \\ 
     \{x,t\} &= \frac{\p}{2} (x-t)(y-z), \; \{y,z\} = \frac{\p}{2} (x+t)(y+z), \\
     \{y,t\} &= \frac{\p}{2} \left( -t(x-t) +y(y+z) \right), \; \{z,t\} = \frac{\p}{2} \left( t(x-t) -z(y+z) \right). 
     \label{eq:poisellip}
\end{split}
\end{equation}
The modular character of $\g^*$ can be directly computed from \eqref{eq:dualell1} and reads
\begin{equation}
    \chi_{\g^*} = -4 \p P_1 ,
\end{equation}
while the left- and right-invariant vector fields associated to it are given by
\begin{equation}\label{eq:elip}
    \lvec{\chi_{\g^*}} = 2 \p \left( y \frac{\partial}{\partial x} - x \frac{\partial}{\partial y} + t \frac{\partial}{\partial z} - z \frac{\partial}{\partial t} \right), \qquad
     \rvec{\chi_{\g^*}} = 2 \p \left( -z \frac{\partial}{\partial x} - t \frac{\partial}{\partial y} + x \frac{\partial}{\partial z} + y \frac{\partial}{\partial t} \right) .
\end{equation}
The multiplicative vector field $\H_\nu^\Pi$ is then given by
\begin{equation}
    \H_\nu^\Pi = \p \left( - (y+z) \frac{\partial}{\partial x} + (x-t) \frac{\partial}{\partial y} + (x-t) \frac{\partial}{\partial z} + (y+z) \frac{\partial}{\partial t} \right) .
    \label{eq:Mellip}
\end{equation}
If we suppose that the modular vector field $\M_\nu^\Pi$ is a Hamiltonian vector field, then $\H_\nu^\Pi$ is also Hamitonian. 

To show that $\W$ is not Hamiltonian we need to see that there is not a function $\mu \colon SL(2,\mathbb R) \to \mathbb R$ such that $\W=\Pi^\sharp _G(d \mu)$. A simple computation shows that
\begin{equation}
    \Pi^\sharp _G(d \mu)  (y+z) = - \frac{\p}{2} (y+z) \varphi(x,y,z,t) , \qquad \W(y+z) = 2 \p (x-t) , 
    \label{eq:ipiell1}
\end{equation}
and 
\begin{equation}
    \Pi^\sharp _G(d \mu)  (x-t) = - \frac{\p}{2} (x-t) \varphi(x,y,z,t) , \qquad \W(x-t) = - 2 \p (y+z) , 
    \label{eq:ipiell2}
\end{equation}
where 
\begin{equation}
    \varphi(x,y,z,t) := (x+t) \left( \frac{\partial \mu}{\partial z} - \frac{\partial \mu}{\partial y} \right) + (y-z) \left( \frac{\partial \mu}{\partial x} + \frac{\partial \mu}{\partial t} \right) .
\end{equation}
Now, from \eqref{eq:ipiell1} we have that 
\begin{equation}
    4 (x-t) = -\varphi(x,y,z,t) (y+z) 
\end{equation}
while from \eqref{eq:ipiell2} we obtain
\begin{equation}
    (x-t) \varphi(x,y,z,t) = 4 (y+z) .
\end{equation}
Therefore, these equations imply that $\varphi(x,y,z,t)^2 = -16$, so it cannot exist a Hamiltonian function $\mu$ for the vector field $\H_\nu^\Pi$ and, as a consequence, for the modular vector field  $\M_\nu^\Pi$. It is important to remark that this argument shows that none of the three Poisson quotients for the elliptic Poisson-Lie structure (second column from Table \ref{table:sl2PHS}) are unimodular as Poisson homogeneous spaces, not only that they are not multiplicative unimodular.

\subsubsection{Parabolic Poisson-Lie structure}

Finally, consider the parabolic Poisson-Lie structure defined on $SL(2,\mathbb R)$ by the classical $r$-matrix
\begin{equation}
r = \p J_{12} \wedge (P_1 + P_2) = \frac12 \p J_3 \wedge J_+ ,
\label{eq:rpar}
\end{equation}
where $\p \in \mathbb{R}-\{0\}$. In the basis \eqref{eq:sl2basis1}, the Lie bialgebra cocycle reads
\begin{equation}
\delta (J_{12}) = \p J_{12} \wedge (P_1 + P_2), \qquad \delta (P_1) = \p P_1 \wedge P_2, \qquad \delta (P_2) = - \p P_1 \wedge P_2 ,
\end{equation}
and the dual Lie algebra takes the form
\begin{equation}
[P^1,P^2]_{\mathfrak{g}^*}= \p (P^1 - P^2) , \qquad 
[P^1, J^{12}]_{\mathfrak{g}^*}=- \p J^{12}, \qquad 
[P^2,J^{12}]_{\mathfrak{g}^*}= - \p J^{12} .
\label{eq:dualpar1}
\end{equation}
Similarly, using the basis \eqref{eq:sl2basis2}, we obtain the Lie bialgebra structure 
\begin{equation}
\delta (J_3) = \p J_3 \wedge J_+, \qquad \delta (J_+) = 0 , \qquad \delta (J_-) = - \p J_+ \wedge J_- ,
\label{eq:dualpar2}
\end{equation}
and the dual Lie algebra
\begin{equation}
[J^3, J^+]_{\mathfrak{g}^*}= \p J^3, \qquad [J^3, J^-]_{\mathfrak{g}^*}= 0 ,\qquad [J^+,J^-]_{\mathfrak{g}^*}= - \p J^- .
\end{equation}
Similarly to the previous cases, simply using Theorem \ref{th:uni-perp}, we can directly check that the two coisotropic quotients, which in this case are the Anti-de Sitter space (two-sheeted hypeboloid) and the one-sheeted hyperboloid are not unimodular since $\chi_{\h^0} \neq 0$. Also, given that $d^\g (\lambda J^3 \wedge J^{-}) = 0$ and $\theta_0 = 0$ is the only admissible cocycle on $SL(2,\mathbb R)$, condition $ii)$ from Theorem \ref{th:semipreservediff} is not satisfied, and thus this last Poisson quotient $L_2$ (which is again the only Poisson quotient of Poisson subgroup type for this structure) is not multiplicative unimodular.

\paragraph{Not unimodular Poisson homogeneous spaces for the parabolic Poisson-Lie structure.} 

To complete our analysis, we prove that again there are no unimodular Poisson homogeneous spaces for the parabolic Poisson-Lie structure. 

We have that $\chi_{{\mathfrak sl}(2,\R)}=0$  and  $\chi_{{\mathfrak h}}=0$. Therefore, there is an invariant volume form $\nu$ on $G/H$ (see Corollary \ref{cor:formas-invariantes}).
 
We will show that the multiplicative vector field $\H_\nu^\Pi$   is not a Hamiltonian vector field. 

For the coboundary Poisson-Lie structure defined by \eqref{eq:rpar} the fundamental brackets are given by
\begin{equation}
\begin{array}{ll}
     \{x,y\} = \frac{\p}{2} \left( -x(x-t) -y z \right),  &\{x,z\} = \frac{\p}{2} z^2, \\ [10pt]
     \{x,t\} = - \frac{\p}{2} (x-t)z, &\{y,z\} = \frac{\p}{2} (x+t) z, \\ [10pt]
     \{y,t\} = \frac{\p}{2} \left( -t(x-t) +y z \right),  &\{z,t\} = -\frac{\p}{2} z^2 . 
     \label{eq:poispara}
\end{array}
\end{equation}
From \eqref{eq:dualpar1} or \eqref{eq:dualpar2} we have that 
\begin{equation}
    \chi_{\g^*} = -2 \p J_+ .
\end{equation}
Its left- and right-invariant vector fields are
\begin{equation}
    \lvec{\chi_{\g^*}} = -2 \p \left( x \frac{\partial}{\partial y} + z \frac{\partial}{\partial t} \right), \qquad
    \rvec{\chi_{\g^*}} = - 2 \p \left( z \frac{\partial}{\partial x} + t \frac{\partial}{\partial y} \right) ,
\end{equation}
and therefore the multiplicative vector field $\W = \frac{1}{2} (\rvec{\chi_{\g^*}} - \lvec{\chi_{\g^*}})$ reads
\begin{equation}
    \W = \p \left( - z \frac{\partial}{\partial x} + (x-t) \frac{\partial}{\partial y} +z \frac{\partial}{\partial t} \right) .
\end{equation}
To see that there is not a function $\mu \colon SL(2,\mathbb R) \to \mathbb R$ such that $\W=\Pi^\sharp _G(d \mu)$, we compute 
\begin{equation}
    \Pi^\sharp _G(d \mu)  (z) = \frac{\eta}{2} \varphi(x,y,z,t), \qquad \W(z) = 0,
\end{equation}
where we have written
\begin{equation}
    \varphi(x,y,z,t) := (x+t) \frac{\partial \mu}{\partial y} + z \left(\frac{\partial \mu}{\partial x} + \frac{\partial \mu}{\partial t}  \right) .
\end{equation}
Since $\Pi^\sharp _G(d \mu)  (z) = \W(z),$ we have that $\varphi(x,y,z,t) = 0$. However, we have that
 \begin{equation}
     \Pi_G^{\#}(d\mu) (x-t) = \frac{\eta}{2} (x-t) \varphi(x,y,z,t), \qquad \W(x-t) = -2 \p z ,
\end{equation}
and therefore $\Pi^\sharp _G(d \mu)  (x-t) \neq \W(x-t)$, so $\W$ cannot be a Hamiltonian vector field. 

Some last comments regarding our examples on semi-simple groups are in order. Firstly, we have shown that none of the Poisson quotients considered are unimodular. We have done so by performing the explicit computation at the Lie group level and showing that the multiplicative vector field $\W$ that projects to the modular vector field on the quotient cannot be Hamiltonian, since in every case contradictions appear when trying to solve associated the system of PDEs, and therefore no Hamiltonian function exists for $\W$. Secondly, quite remarkably, in all these examples Theorem \ref{th:uni-perp} allows to conclude that none of the the coisotropic quotients are unimodular, but it fails to detect this feature when the quotient is of Poisson subgroup type, and thus we need to use Theorem \ref{th:semipreservediff}, which indeed is sufficient to infinitesimally check that none of these Poisson quotients are multiplicative unimodular. Finally, we can end this Section by summing up these results in the following theorem

\begin{theorem}
    There are no unimodular coisotropic Poisson homogeneous spaces for the groups $SO(3)$ and $SL(2,\mathbb R)$.
    \label{th:semisimplenotunim}
\end{theorem}

\section{Invariant volumen forms for Hamiltonian systems on coisotropic Poisson homogeneous spaces}\label{sec:preservation}

As we have previously mentioned, an arbitrary Poisson manifold is unimodular if and only if there exists a volume form which is preserved by all Hamiltonian vector fields (see, for instance, \cite{gutierrez2023unimodularity}). 

In the particular case of a connected Poisson-Lie group $(G,\Pi)$ with unimodular dual Lie algebra $\g^*$, from an arbitrary volume form on $G$, one can consider a new  volume form which  is preserved by all Hamiltonian vector fields. A converse of this result is given in \cite{gutierrez2023unimodularity} (see Theorem 3.7 in \cite{gutierrez2023unimodularity}) for a particular kind of Hamiltonian functions: the Morse functions at the identity element of the Poisson-Lie group.

In this section we will analyze the relation between the preservation of volume forms and the unimodularity of coisotropic Poisson homogeneous spaces. We will show that in this case the relationship between these two notions is not as simple as the one that exists in the Poisson-Lie group framework.

Let  $\nu$ be a volume form on a coisotropic Poisson homogeneous space $(G/H,\Pi)$ and $h:G/H\to \R$ a Hamiltonian function.  Then there is a function $\sigma\in C^\infty(G)$ and $\V\in \wedge^{m-n}\h^0$ with $\V\not=0$ satisfying 
$$
    q^*\nu =e^\sigma \lvec{\V}
$$
and the conditions $i)$ and $ii)$ given in  Theorem \ref{thm:volume-form}. 

Now, let $\tilde{\nu}$ be a volume form on $G/H$ which is preserved. Then, there is a function $\tau:G/H\to \R$ such that $\tilde\nu=e^\tau\nu$, and the preservation of the volume forms implies 
    \begin{equation}\label{preservation}
    0= \mathcal L_{X_h^{\Pi}} (e^{\tau} \nu) = e^{\tau} (X_h^{\Pi} (\tau) + \M_\nu^\Pi (h) ) \nu = e^{\tau} (-X_\tau^{\Pi} (h) + \M_\nu^\Pi (h) ) \nu,
    \end{equation}
where $\M_\nu^\Pi$ is the modular vector field of $\Pi$ with respect to $\nu.$ 
Then, using \eqref{eq:campo:simplificado}, Proposition \ref{cois} and that $q^*$ is an injective morphism, we deduce the following result 
\begin{proposition}\label{pro:pres}
    Let $(G/H,\Pi)$ be a coisotropic Poisson homogeneous space and $h:G/H\to \R$  a Hamiltonian function. Then, the
Hamiltonian vector field $X_h^\Pi$  preserves a volume form on $G/H$ if and only if there exist a function $\tau:G/H\to \R$ and a function $\sigma:G\to \R$ such that 
\begin{equation}\label{preH}
 X_{h\circ q}^{\Pi_G}(\sigma+\tau\circ q)  +\frac12 \left( \rvec{\chi_{\g^*}}(h\circ q) -\lvec{\chi_{\g^*}}(h\circ q)   - \rvec{\chi_\g}(X^{\Pi_G}_{h\circ q})\right)=0 .
\end{equation}
In such a case, the preserved volume form $\tilde{\nu}$ satisfies   
$$q^*(\tilde\nu)=e^{(\sigma + \tau\circ q)}\lvec{\V},$$
with  $\V\in \wedge^{m-n}\h^0$ such that 
 the conditions $i)$ and $ii)$ (for $\sigma$ and $\V$) from  Theorem \ref{thm:volume-form} hold.
\label{prop:presvol}  
\end{proposition}

\begin{example}
    Consider the Poisson quotient $H_2 \times \mathbb Z_2$ corresponding to the elliptic Poisson-Lie structure on $SL(2,\mathbb R)$ (second row and second column from Table \ref{table:sl2PHS}). Consider the Hamiltonian function $\hat h: SL(2,\R)\to \R$ given by 
    \begin{equation}
        \hat h = \frac{1}{2} \langle A,A \rangle = \frac{1}{2} \mathrm{Tr} (A^t A) = \frac{1}{2}(x^2+y^2+z^2+t^2).
    \end{equation}
    We have that $\hat h$ is $q$-basic since 
    \begin{equation}
        \lvec{P_1} (\hat h) = \frac{1}{2} \bigg( y \frac{\partial \hat h}{\partial x} - x \frac{\partial \hat h}{\partial y} + t \frac{\partial \hat h}{\partial z} - z \frac{\partial \hat h}{\partial t} \bigg) = 0 .
    \end{equation}

    We recall that $\chi_\g=0$ and that, since $SL(2,\R)$ has an invariant volume form, the function $\sigma:G\to \R$ on $SL(2,\R)$ in the equation \eqref{preH} is the zero function. 
    
    On the other hand, using \eqref{eq:elip}, we deduce 
    $$\rvec{\chi_{\g^*}}(\hat h) -\lvec{\chi_{\g^*}}(\hat h)=0.$$
     Therefore,  from Proposition \ref{pro:pres},  we have that the Hamiltonian vector field $X^\Pi_h$ preserves the (invariant) volume form $\nu$, where $\widehat{h}=h\circ q$ and $q:SL(2,\R)\to H_2\times {\Bbb Z}_2$ is the canonical projection.
\end{example}

Next, we will analyze the implications of the existence of a volume form which is preserved by the Hamiltonian vector field for a particular kind of Hamiltonian functions: the $H$-Morse functions at the identity element of the Poisson homogeneous space.     

\begin{definition}
Let $G/H$ be a homogeneous space. A function $h\in C^\infty(G/H)$ is said to be $H$-Morse at $eH$ if 
\begin{itemize}
\item[i)] $eH$ is a singular point of $h$, that is, $d h (eH)=0$.
\item[ii)] $(\textrm{Hess } h)(eH)$ is nondegenerate, with $(\textrm{Hess } h)(eH) \colon \g/\h \times \g/\h \to \R$  the
Hessian of $h$ at $eH$, i.e., the symmetric bilinear form on $\g/\h$ given by
\begin{equation}\label{eq:Hessian}
(\textrm{Hess } h)_{eH}(X ,Y )= X (V^{Y }(h) ), \qquad X,Y \in \g/\h\cong T_{eH}(G/H),
\end{equation}
where $V^{Y}\in \mathfrak{X}(G/H)$ is an arbitrary vector field such that $V^{Y}(eH)=Y$.
\end{itemize}
\end{definition}


Note that in the particular case when $H=\{e\}$, that is, the homogeneous space is the Lie group $G,$ we recover the definition of a Morse function in $G$ at $e$ (see Definition 3.5 in \cite{gutierrez2023unimodularity}). 
\begin{theorem}
    Let $(G/H,\Pi)$ be a coisotropic Poisson homogeneous space and $h \in \mathcal{C}^\infty (G/H)$ a $H$-Morse function at $eH$. If there exists a volume form $\nu$ on $G/H$ which is preserved by $X_h^{\Pi}$, i.e. $\mathcal L_{X_h^{\Pi}}\; \nu=0$, then  ${\h^0}$ is unimodular. 
    \label{th:morseh0}
\end{theorem}

\begin{proof}
    From the preservation of $\nu$, and using (\ref{14'}) and (\ref{eq:modular-class}),
    we deduce  that $\M_\nu^\Pi (h)=0.$
    
    From the definition of the Hessian, for any $Y\in \g/\h\cong T_e(G/H)$, we obtain that
    \begin{equation}
    \begin{split}
    (\textrm{Hess } h)_{eH} (Y, \M_\nu^\Pi (eH) ) &
    = Y(\M_\nu^\Pi (h))= 0.
    \end{split}
    \end{equation}
     Using now that $h$ is $H$-Morse at $eH$ (and therefore its Hessian is non-degenerate at $eH$), we obtain that $\M_\nu^\Pi (eH) = 0$. From equation \eqref{eq:proyeccion} and since $\Pi_G(e)=0,$ we know that $\M_\nu^\Pi (eH) = q_* (x_{\h^0})$ and thus $x_{\h^0} \in \h.$ But, $\chi_{\h^0}$ is the projection of $x_{\h^0}$ via the dual map $\iota^*:\g\to (\h^0)^*$ of the inclusion $\iota: \h^0\to \g^*.$ Thus,  we conclude that the condition  $x_{\h^0} \in \h$ implies $\chi_{\h^0} = 0$, i.e. $\h^0$ is unimodular.
\end{proof}

In the particular case when $H=\{e\}$ (that is, the Poisson homogeneous space $G/H$ is the Poisson-Lie group $G$ ), we deduce that $\h^0=\g^*$ is unimodular. So, $G$ is unimodular  and we recover Theorem 3.7 in  \cite{gutierrez2023unimodularity}. For the more general case of a Poisson homogeneous space, the situation is more complicated. In fact, we will present  examples that show that the existence of a $H$-Morse function on a coisotropic Poisson homogeneous space, with $\h^0$ unimodular, does not imply that the Poisson structure on the quotient is unimodular.

\begin{example}\label{Ex:6.5}
Recall the `subgroup sphere' described in Table \ref{table:so3PHS}. As we have proved above, this Poisson homogeneous space is not unimodular although $\mathfrak h^0$ is unimodular. 

Consider the function
\begin{equation}
    \begin{split}
        \hat h : SO(3) &\to \R \\
        (x,y,z,t) &\to (x^2+y^2)(z^2+t^2),
    \end{split} 
\end{equation}
where $x,y,z,t$ are coordinate functions on $SO(3)$ as in (\ref{eq:SU2coords}). In particular, the identity element is $e=(1,0,0,0).$ Moreover, $\widehat{h}$ 
 is $q$-basic since
\begin{equation}
    \lvec{J_3} (\hat h) = \frac{1}{2} \bigg(- y \frac{\partial \hat h}{\partial x} + x \frac{\partial \hat h}{\partial y} - t \frac{\partial \hat h}{\partial z} + z \frac{\partial \hat h}{\partial t} \bigg) = 0 ,
\end{equation}
and therefore, it induces the function on the quotient $ h : S^2 \simeq SO(3)/S^1 \to \R$, with $ h \circ q = \hat h$. 

The Hamiltonian vector field of 
$h$ preserves a volume form. In fact, from \eqref{lxg*} and that $\g$ is unimodular, we have 

\begin{equation}\label{preH2}
   \rvec{\chi_{\g^*}}(h\circ q) -\lvec{\chi_{\g^*}}(h\circ q)   - \rvec{\chi_\g}(X^{\Pi_G}_{h\circ q})=\eta \bigg( - t \frac{\partial \hat h}{\partial z} + z \frac{\partial \hat h}{\partial t} \bigg)=0. 
\end{equation}
Since $SO(3)$  is semisimple,  there are not non-trivial 1-cocycles and therefore $\sigma=0$ (see Remark \ref{rem-2.7}). Taking $\tau=0$ in \eqref{preH} and using \eqref{preH2}, we deduce that $X_h^\Pi$ preserves a volumen form on $G/H$. 

Moreover, $h$ is $H$-Morse since 
\begin{enumerate}
    \item[i)] $d \hat h = 2 ((z^2+t^2)(xdx+ydy)+(x^2+y^2)(zdz+tdt))$, and then $d \hat h (e)=0$. Thus, $q^*(d {h} (eH))=0$. Since $q^*$ is injective, then $d {h} (eH)=0.$
    \item[ii)] Taking into account that $T_{eH} (SO(3)/S^1) \simeq \langle J_1,J_2 \rangle$, we have that
    \begin{equation}
        \begin{split}
            (\mathrm{Hess} \;  h)_{eH} (J_1,J_1) &= 1/2, \\
            (\mathrm{Hess} \;  h)_{eH}(J_1,J_2) &= (\mathrm{Hess} \;  h)_{eH} (J_2,J_1) = 0, \\
            (\mathrm{Hess} \;  h)_{eH} (J_2,J_2) &= 1/2, \\
        \end{split}
    \end{equation}
    i.e.,
    \begin{equation}
        [(\mathrm{Hess } \;  h)_{eH}]_{\{J_1,J_2\}} = 
        \begin{pmatrix}
            1/2 & 0 \\
            0 & 1/2 \\
        \end{pmatrix}
    \end{equation}
    and the Hessian is not degenerate at $eH$. 
\end{enumerate}
\end{example}

Next, we will present a higher dimensional example that shows that, even when the dynamical system is completely integrable, the preservation of a global volume form is by no means  guaranteed. 
\begin{example}\label{Ex:6.6}
    Consider the Lie group 
    \begin{equation}
        G = SL (n,\mathbb R) = \{ A \in \mathrm{GL} (n, \mathbb R) \, | \, \det A = 1 \}
    \end{equation}
    and the Poisson-Lie group structure given by the following fundamental brackets 
    \begin{equation}
        \{a_{ij},a_{kl} \} = ((-1)^{(i-k)} - (-1)^{(l-j)} )a_{il} a_{kj} . 
    \end{equation}
    This Poisson-Lie group structure is the one associated to the standard or Drinfel'd-Jimbo $r$-matrix (see \cite{fernandes1994completely} and references therein). In order to describe the dual Lie algebra, let us introduce the non-degenerate bilinear form $\langle \cdot , \cdot \rangle : \mathfrak{sl}(n,\mathbb R) \times \mathfrak{sl}(n,\mathbb R) \to \mathbb R$ given by $\langle A , B \rangle = \Tr (A B)$. This bilinear form allows us to identify $\mathfrak{sl}(n,\mathbb R)^* \simeq \mathfrak{sl}(n,\mathbb R)$ as vector spaces, and the dual Lie algebra can therefore be described as a second Lie bracket on $\mathfrak{sl}(n,\mathbb R)$ given by $[A,B]_* = [R A,B] + [A,R B]$, where $R : \mathfrak{sl}(n,\mathbb R) \to \mathfrak{sl}(n,\mathbb R)$ is the following map
    \begin{equation}
    R A = 
        \begin{cases}
            - A &\text{ if } \quad A \in \mathfrak t_+ \\
            0 &\text{ if } \quad A \in \mathfrak d \\
            A &\text{ if } \quad A \in \mathfrak t_- \\
        \end{cases} ,
        \label{eq:Rsln}
    \end{equation}
    where $\mathfrak d, \mathfrak t_+, \mathfrak t_-$ are the sets of diagonal, upper and lower triangular matrices with zero trace, respectively.

Consider now the closed subgroup $H\kern-3pt=\kern-3pt SO(n,\mathbb R)$ and the homogeneous space $G/H\kern-3pt=\kern-3pt SL (n,\mathbb R)/SO(n,\mathbb R)$. Note that for $n=2$, this quotient corresponds to the one presented in the second row, first column of Table \ref{table:sl2PHS}. In order to describe this quotient, we decompose the semisimple Lie algebra $\mathfrak{sl}(n,\mathbb R)$ as $\mathfrak{sl}(n,\mathbb R) = \mathfrak d \oplus \mathfrak s \oplus \mathfrak q$ (as a vector space), where $\mathfrak d$ are diagonal matrices with zero trace, while $\mathfrak s, \mathfrak q$ are symmetric and skew-symmetric matrices with zeroes at the main diagonal entries. We have that $\dim \mathfrak d = n-1$, $\dim \mathfrak s = \frac{n(n-1)}{2}$ and $\dim \mathfrak q = \frac{n(n-1)}{2}$. Consider the basis of $\mathfrak{sl}(n,\mathbb R)$, adapted to the quotient above, given by 
    \begin{equation*}
    \begin{split}
    D_{i} &:= E_{ii}-E_{i+1,i+1}, \qquad 1 \leq i \leq n -1, \\
    S_{ij} &:= E_{ij}+E_{ji}, \qquad 1 \leq i < j \leq n, \\
    Q_{ij} &:= E_{ij}-E_{ji}, \qquad 1 \leq i < j \leq n, \\
    \end{split}
    \end{equation*}
    where $E_{ij}$ is the $n \times n$ matrix with entries 1 in the position $i,j$ and 0 elsewhere. In terms of this basis, the commutation relations of $\mathfrak{sl}(n,\mathbb R)$ read
    \begin{equation*}
    \begin{split}
    [Q_{ij},Q_{kl}] &= \delta_{il} Q_{jk} + \delta_{jk} Q_{il} - \delta_{jl} Q_{ik} - \delta_{ik} Q_{jl}, \\
    [Q_{ij},S_{kl}] &= -\delta_{il} S_{jk} + \delta_{jk} S_{il} + \delta_{jl} S_{ik} - \delta_{ik} S_{jl}, \\
    [S_{ij},S_{kl}] &= \delta_{il} Q_{jk} + \delta_{jk} Q_{il} + \delta_{jl} Q_{ik} + \delta_{ik} Q_{jl}, \\
    [D_k,Q_{ij}] &= (\lambda_i - \lambda_j) S_{ij}, \\
    [D_k,S_{ij}] &= (\lambda_i - \lambda_j) Q_{ij}, \\
    [D_i,D_j] &= 0, \\
    \end{split}
    \end{equation*}
    where $\lambda_i$ are defined, for any diagonal matrix $D=(d_1,\ldots,d_n)$, by $\lambda_i (D) = \lambda_i (d_1,\ldots,d_n) = d_i$. Note that $\h = \mathfrak{so}(n, \mathbb R) = \langle Q_{ij} \rangle$.

    Since both $SL (n,\mathbb R)$ and $SO(n,\mathbb R)$ are unimodular Lie groups, we know that there exist invariant volume forms on $SL (n,\mathbb R)/SO(n,\mathbb R)$ and they cannot be semi-invariant (see Proposition \ref{cor:caract:semi-invariante} and Corollary \ref{cor:formas-invariantes}).

    In terms of this basis, the map \eqref{eq:Rsln} reads 
    \begin{equation}
    R (Q_{ij}) = -S_{ij}, \qquad R (S_{ij}) = -Q_{ij}, \qquad R (D_{i}) = 0 
    \end{equation}
    and the dual Lie algebra is given by
    \begin{equation*}
    \begin{split}
    [Q_{ij},Q_{kl}]_* &= -[S_{ij},Q_{kl}] - [Q_{ij},S_{kl}] = 2(\delta_{il} S_{jk} - \delta_{jk} S_{il}), \\
    [S_{ij},Q_{kl}]_* &= -[Q_{ij},Q_{kl}] - [S_{ij},S_kl] = 2(\delta_{il} Q_{kj} - \delta_{jk} Q_{il}), \\
    [S_{ij},S_{kl}]_* &= -[Q_{ij},S_{kl}] - [S_{ij},Q_{kl}] = 2 (\delta_{il} S_{kj} - \delta_{jk} S_{il}), \\
    [D_k,Q_{ij}]_* &= -[D,S_{ij}] = -(\lambda_i - \lambda_j) Q_{ij}, \\
    [D_k,S_{ij}]_* &= - [D,Q_{ij}]= - (\lambda_i - \lambda_j) S_{ij}, \\
    [D_i,D_j]_* &= 0. \\    
    \end{split}
    \end{equation*}
    Note that $T_e(SL (n,\mathbb R)/SO(n,\mathbb R)) \simeq \h^0 \simeq \mathfrak{so}(n, \mathbb R)^0 \simeq \langle D_k, S_{ij} \rangle$.
    
    From these commutation relations, we can see that $\chi_{\h^0} = - 2 \displaystyle\sum_{k=1}^{n-1} D_k = -2 (E_{11}-E_{nn})$, so the Poisson quotient $(SL (n,\mathbb R)/SO(n,\mathbb R), \Pi)$ is not unimodular by Theorem \ref{th:h0unim}. Moreover, we have that 
    \begin{equation*}
        \chi_{\g^*} = - 4 \sum_{k=1}^{n-1} D_k = -4 (E_{11}-E_{nn}) .
    \end{equation*}
    From this expression we can easily compute
    \begin{equation*}
    \lvec{\chi_{\g^*}} = -4 \sum_{i=1}^n \bigg( a^{i1} \frac{\partial}{\partial a^{i1}} - a^{in} \frac{\partial}{\partial a^{in}} \bigg), \qquad 
    \rvec{\chi_{\g^*}} = -4 \sum_{i=1}^n \bigg( a^{1i} \frac{\partial}{\partial a^{1i}} - a^{ni} \frac{\partial}{\partial a^{ni}} \bigg),
    \end{equation*}
    and therefore the horizontal modular class of $(SL (n,\mathbb R)/SO(n,\mathbb R), \Pi)$ associated to an invariant volume form $\nu$ is given by
    \begin{equation*}
    \W = \frac{1}{2} (\rvec{\chi_{\g^*}} - \lvec{\chi_{\g^*}}) = 4 \bigg( a^{1n} \frac{\partial}{\partial a^{1n}} - a^{n1} \frac{\partial}{\partial a^{n1}} \bigg)
    -2 \sum_{i=2}^{n-1} 
    \bigg( a^{1i} \frac{\partial}{\partial a^{1i}} - a^{i1} \frac{\partial}{\partial a^{i1}} - a^{ni} \frac{\partial}{\partial a^{ni}} + a^{in} \frac{\partial}{\partial a^{in}} \bigg) .
    \end{equation*}
    The Hamiltonian function
    \begin{equation}
        \begin{split}
           \hat h : SL (n,\mathbb R) &\to \mathbb R \\
            A\qquad  &\to \mathrm{Tr}(AA^T)= \sum_{1\leq i,j \leq n} a_{ij}^2
        \end{split}
        \label{eq:todah}
    \end{equation}
    on the Poisson-Lie group, defines the dynamics of the Toda lattice, a well-known completely integrable system. The Hamiltonian vector field reads
    \begin{equation*}
        X_{\hat h}^{\Pi_G} = \sum_{1 \leq i,j,k,l \leq n} ((-1)^{i-k} - (-1)^{l-j}) a_{ij} a_{il} a_{kj} \frac{\partial}{\partial a_{kl}} .
    \end{equation*}

    The function $\hat{h}$ induces a Hamiltonian function $h:SL(n,\R)/SO(n)\to \R$ on the Poisson homogeneous space such that $\hat{h}=h\circ q.$ Note that 
    $$\hat{h}(SO(n,\R))=\hat{h}(\mathrm{Id}_{n,n}).$$

    
    In order to see that the vector field $X_{h}^{\Pi}$ does not preserve any volume form, consider the point 
    \begin{equation*}
        g = \begin{pmatrix}
            0 & a & 0 \\
            -\frac{1}{a} & 0 & 0 \\
            0 & 0 & \mathrm{Id}_{n-2,n-2} \\
        \end{pmatrix} \in SL (n,\mathbb R),
    \end{equation*}
    with $a \in \mathbb R - \{0\}$,
    and write 
    \begin{equation*}
    \begin{split}
        X_{\hat h}^{\Pi_G} (g) &= \sum_{1 \leq j,k,l \leq n} ((-1)^{1-k} - (-1)^{l-j}) a_{1j} a_{1l} a_{kj} \\
        &+ \sum_{1 \leq j,k,l \leq n} ((-1)^{2-k} - (-1)^{l-j}) a_{2j} a_{2l} a_{kj} \\
        &+ \sum_{\substack{i \geq 3 \\ 1 \leq j,k,l \leq n}} ((-1)^{i-k} - (-1)^{l-j}) a_{ij} a_{1l} a_{kj} \\
        &= ((-1)^{1-1} - (-1)^{2-2}) a^3 - ((-1)^{2-2} - (-1)^{1-1}) \frac{1}{a^3} +  \sum_{i \geq 3} (-1)^{i-i} - (-1)^{i-i}) a_{ii}^3 = 0,
    \end{split}
    \end{equation*}    
    and thus $g$ is a singular point of $X_{\hat h}^{\Pi_G}.$ 
    
    Moreover, we have that
    \begin{equation*}
        \W (h) = 4 \bigg(2 \big((a^{1n})^2 - (a^{n1})^2 \big) - \sum_{i=2}^{n-1} (a^{1i})^2 - (a^{i1})^2 - (a^{ni})^2 + (a^{in})^2 \bigg) ,
    \end{equation*}
    and thus 
    \begin{equation*}
        (\W (h)) (g) = -\frac{4}{a^2} (a^2+1)(a+1)(a-1) ,
    \end{equation*}
    so, for any $a \not \in \{-1,1\}$, $(\W (h)) (g) \neq 0.$ Therefore, the equality \eqref{preH} is not possible for any function $\tau$ on $SL(n, \mathbb R)/SO(n, \mathbb R).$
    
    In conclusion, we have that the Toda lattice Hamiltonian \eqref{eq:todah} (which is a well-known completely integrable system) projects on  $SL(n, \mathbb R)/SO(n, \mathbb R).$ Its projection is the Hamiltonian vector field $X_{\hat h}^{\Pi_G}$ (which is a direct consequence of Proposition \ref{cois}) but, from Proposition \ref{prop:presvol}, this last dynamics does not preserve any volume form on the Poisson homogeneous space $(SL (n,\mathbb R)/SO(n,\mathbb R),\Pi)$.
    
\end{example}




In Examples \ref{ex:onlysemi1} and \ref{ex:onlysemi2}, we presented the infinitesimal description of a Poisson homogeneous space which admits semi-invariant but not invariant volume forms and, in addition, the Poisson structure is multiplicative unimodular. 

Next, we will give a global description of the Poisson homogeneous space and, moreover, we will see that a compartmental epidemiological  model (see \cite{BBG2020physicaD}) is Hamiltonian with respect to this Poisson structure. Then, it is clear that the dynamics preserves a semi-invariant volume form and we will obtain an explicit description of such a volume form. 
\begin{example}\label{Ex:6.7}
    Consider the dynamical system defined by
    \begin{equation}
        \dot x^1 = 1-x^1, \qquad \dot x^2 = x^1-1-x^3, \qquad \dot x^3 = x^3 .
        \label{eq:compart}
    \end{equation}
    Since $\dot x_1 + \dot x_2 + \dot x_3 = 0$, this is a compartmental model and therefore it admits a Hamiltonian description (see \cite{BBG2020physicaD}), with Hamiltonian function given by 
    \begin{equation}
        h = x_1+x_2+x_3.
        \label{eq:hamexapmleunim}
    \end{equation} 
    Now, we will see that this system can be described in terms of the multiplicative unimodular Poisson Hamiltonian system given in Example \ref{ex:onlysemi2}.

    We can parameterise the Lie group $G$ associated to the Lie algebra $\g$ given in Example \ref{ex:onlysemi1} using exponential coordinates of the second kind as
    \begin{equation}
        G = \exp (x^1 X_1) \exp (x^2 X_2) \exp (x^3 X_3) \exp (x^4 X_4). 
    \end{equation}
    In terms of these coordinates, the group multiplication reads
    \begin{equation}
        (x^1,x^2,x^3,x^4) \circ (y^1,y^2,y^3,y^4) = (x^1+y^1,x^2+y^2,x^3+e^{x^4} y^3,x^4+y^4).
    \end{equation}
    From here, the left-invariant vector vector fields can be directly computed
    \begin{equation}
        \lvec X_1 = \frac{\partial}{\partial x^1}, \qquad \lvec X_2 = \frac{\partial}{\partial x^2}, \qquad \lvec X_3 = e^{x^4} \frac{\partial}{\partial x^3}, \qquad \lvec X_4 = \frac{\partial}{\partial x^4}.  
    \end{equation}
    Following a standard integration procedure (see \cite{Vaisman1994poissonbook}), it can be shown that the Poisson--Lie structure which integrates the Lie bialgebra given in Example \ref{ex:onlysemi2} reads
    \begin{equation}
        \{x_1,x_2\} = x_1 - 1, \qquad \{x_1,x_3\} = 0, \qquad \{x_2,x_3\} = x_3, \qquad \{x_4,\cdot\} = 0 .
    \end{equation}
    In Example \ref{ex:onlysemi2} it was proved that $G/H$, where $\h=Lie \, H = \langle X_4 \rangle$, is a Poisson quotient. The Poisson structure may be explicitly described as follows 
    \begin{equation}
        \{x_1,x_2\} = x_1 - 1, \qquad \{x_1,x_3\} = 0, \qquad \{x_2,x_3\} = x_3 .
        \label{eq:PHSexapmleunim}
    \end{equation}   
    We can recover the compartmental model \eqref{eq:compart} by means of this Poisson structure \eqref{eq:PHSexapmleunim} and the Hamiltonian function \eqref{eq:hamexapmleunim}
    \begin{equation}
        \begin{split}
            \dot x^1 &= \{h, x^1 \} = 1-x^1, \\ 
            \dot x^2 &= \{h, x^2 \} = x^1-1-x^3,\\ 
            \dot x^3 &= \{h, x^3 \} = x^3. 
        \end{split} 
    \end{equation}
    Therefore, this dynamical system preserves a semi-invariant volume form. This volume form is given by $q^* \nu = e^{\sigma} \lvec \V$, with $\V = X^1 \wedge X^2 \wedge X^3$ and $\sigma : G \to \mathbb R$ such that $d \sigma = \lvec X^4$. The left-invariant one-forms on $G$ are given by 
    \begin{equation}
        \lvec X^1 = d x^1, \qquad \lvec X^2 = d x^2, \qquad \lvec X^3 = e^{-x^4} d x^3, \qquad \lvec X^4 = d x^4,
    \end{equation}
    and therefore $\sigma = x^4$ and  $q^* \nu = d x^1 \wedge d x^2 \wedge d x^3$. The Hamiltonian vector field for the system \eqref{eq:compart} reads
    \begin{equation}
        X_h^\Pi = (1-x^1) \frac{\partial}{\partial x^1} + (x^1-1-x^3) \frac{\partial}{\partial x^2} + x^3 \frac{\partial}{\partial x^3} ,
    \end{equation}
    and it is clear that $\mathcal L_{X_h^\Pi} (d x^1 \wedge d x^2 \wedge d x^3) = d (i_{X_h^\Pi} (d x^1 \wedge d x^2 \wedge d x^3)) = 0$, so the semi-invariant volume form $d x^1 \wedge d x^2 \wedge d x^3$ is preserved.

\end{example}


\section*{Acknowledgements}

I. Gutierrez-Sagredo has been partially supported by Agencia Estatal de Investigaci\'on (Spain) under grant PID2023-148373NB-I00 funded by MCIN /AEI /10.13039/501100011033/FEDER and by the Q-CAYLE Project funded by the Regional Government of Castilla y Le\'on (Junta de Castilla y Le\'on). D. Iglesias, J.C. Marrero and E. Padr\'on acknowledge financial support from the Spanish Ministry of Science and Innovation under grant PID2022-137909NB-C22. All the authors  has been partially supported by Agencia Estatal de Investigaci\'on (Spain) under grant RED2022-134301-TD.  I. Gutierrez-Sagredo thanks Universidad de La Laguna where part of the work has been done for the hospitality and support.

{\bf Data Availability} The data supporting the conclusions of this paper are included within the article

{\bf Conflicts of interest}  There are no conflict of interest in this article.


\begin{thebibliography}{99}

\bibitem{AM}  R. Abraham and J.E. Marsden, \textit{Foundations of Mechanics} (2nd. edition), Benjamin Cummings, Reading, Massachusetts, 1978.

\bibitem{Ar} V. I. Arnold, \textit{Mathematical methods of classical mechanics,} Translated from the Russian by K. Vogtmann and A. Weinstein. Second edition. Graduate Texts in Mathematics, 60. Springer-Verlag, New York, 1989.

\bibitem{BBG2020physicaD} A. Ballesteros, A. Blasco and I. Gutierrez-Sagredo, Hamiltonian structure of compartmental epidemiological models, \textit{Phys. D Nonlinear Phenom.} \textbf{413} (2020) 132656.

\bibitem{BMN2017homogeneous} A. Ballesteros, C. Meusburger and P. Naranjo, AdS Poisson homogeneous spaces and Drinfel'd doubles, \textit{Journal of Physics A: Mathematical and Theoretical} \textbf{50} (2017) 395202.

\bibitem{BoCiStTa1}
F. Bonechi, N. Ciccoli, N. Staffolani and M. Tarlini, On the integration of Poisson homogeneous spaces, \textit{J. Geom. Phys.} \textbf{58} (2008) no. 11, 1519--1529.

\bibitem{BoCiStTa2}
F. Bonechi, N. Ciccoli, N. Staffolani and M. Tarlini, The quantization of the symplectic groupoid of the standard Podle's sphere.  \textit{J. Geom. Phys.}  \textbf{62} (2012), 1851--1865.

 \bibitem{BuIgLu} 
H. Bursztyn, D. Iglesias-Ponte and JH Lu, Dirac geometry and integration of Poisson homogeneous spaces,  \textit{J. Differential Geom.} \textbf{126}(3) (2024) 939--1000. 

\bibitem{RC}R. Caseiro and R.L. Fernandes, The modular class of a Poisson map, \textit{Annales de lnstitut Fourier} \textbf{63}, 4 (2013) 1285--1329. 

\bibitem{ELW}S. Evens J.-H. Lu and A. Weinstein, Transverse measures, the modular class and a cohomology pairing for Lie algebroids, \textit{Quarterly J. Math.} \textbf{50} (1999) 417--436.

\bibitem{Dr} V. G. 
Drinfel'd, Hamiltonian structures on Lie groups, Lie bialgebras and the geometric meaning of classical Yang-Baxter equations, (Russian) \textit{Dokl. Akad. Nauk SSSR} \textbf{268} (1983), no. 2, 285-287. English translation: \textit{Soviet Math. Dokl.} \textbf{27} (1983) no. 1, 68--71.

\bibitem{Dr2} V.G. Drinfel'd, On Poisson homogeneous spaces of Poisson-Lie groups, \textit{Theo. Math. Phys.} \textbf{95} (1993)
226--227.

\bibitem{FGM} Y. N. Fedorov, L. C. Garc\'{\i}a-Naranjo and J. C. Marrero, Unimodularity and Preservation of Volumes in Nonholonomic Mechanics, \textit{Journal of Nonlinear Science} \textbf{25} (2015) 203-246. 

\bibitem{fernandes1994completely} R. L. Fernandes, Completely integrable bi-Hamiltonian systems, \textit{Journal of Dynamics and Differential Equations} \textbf{6} (1994) 53--69.

\bibitem{Rui} R. L. Fernandes, Completely Integrable Bi-Hamiltonian Systems, Ph.D. Thesis, University of Minnesota, 1994. 

\bibitem{Ginzburg} V. Ginzburg, Momentum mapping and Poisson cohomology, \textit{International Journal of Mathematics}, \textbf{7} (1996) 329--358.

\bibitem{Gr} J. Grabowski,  Modular classes of skew algebroid relations. \textit{Transformation Groups} \textbf{17}, (2012) 989--1010.
\bibitem{gutierrez2023unimodularity}
I. Gutierrez-Sagredo, D. Iglesias-Ponte, J.C. Marrero, E. Padr{\'o}n and Z. Ravanpak,  Unimodularity and invariant volume forms for Hamiltonian dynamics on Poisson--Lie groups, \textit{Journal of Physics A: Mathematical and Theoretical} \textbf{56} (2023) 015203.


\bibitem{ILX} D. Iglesias-Ponte, C. Laurent-Gengoux and P. Xu, Universal lifting theorem and quasi-Poisson groupoids,
\textit{J. Eur. Math. Soc.} \textbf{14} (2012) 681--731.

\bibitem{Ka} E. Karolinsky, 
Symplectic leaves on Poisson homogeneous spaces of Poisson-Lie groups, \textit{Mat. Fiz. Anal. Geom.} \textbf{2} (1995) 306--311.

\bibitem{Ko} V. V. Kozlov, The Euler-Jacobi-Lie integrability theorem, \textit{Regul. Chaotic Dyn.} \textbf{18} (2013), no. 4, 329-343.

\bibitem{K88} V. V. Kozlov, Invariant measures of the Euler-Poincar\'e equations on Lie algebras, \textit{Funkt.
Anal. Prilozh.} \textbf{22} 69-70 (Russian); English trans.: \textit{Funct. Anal. Appl.} \textbf{22} (1988) 58--59.

\bibitem{Li} A. Lichnerowicz, Les vari\'et\'es de Poisson et leurs alg\'ebres de Lie associ\'ees, \textit{J. Differential Geometry} \textbf{12} (1977) 253--300. 

\bibitem{Lu} J.-H. Lu, Multiplicative and Affine Poisson Structures on Lie Groups,  Ph.D. University of California, Berkeley, 1990.

\bibitem{Lu08}
J.-H. Lu, A note on {P}oisson homogeneous spaces,
\textit{Poisson geometry in mathematics and physics}, {Contemp. Math.} \textbf{450}, 173-198, American Mathematical Society, Providence, RI, 2008.


\bibitem{M} J.C. Marrero, Hamiltonian dynamics on Lie algebroids, Unimodularity and preservation of volumes, \textit{J. Geom. Mech.} \textbf{2} (3) (2010) 243-263.

\bibitem{Me} C. Meusburger, Poisson-Lie groups and gauge theory, \textit{Symmetry} \textbf{13} (2021) 1324.

\bibitem{Mos} G.D. Mostow, Homogeneous Spaces With Finite Invariant Measure, 
\textit{Annals of Math.} \textbf{75}, No. 1 (1962) 17--37.

\bibitem{Reyman1996} A. G. Reyman, Poisson Structures Related to Quantum Groups. In L. Castellani \& J. Wess (Eds.), Quantum Groups and Their Applications in Physics, Proceedings of the International School of Physics ``Enrico Fermi'' \textbf{127} (1996) 407--443. 

\bibitem{Roytenberg2002} D. Roytenberg, Poisson Cohomology of SU(2)-Covariant ``Necklace'' Poisson Structures on $S^2$,  \textit{Journal of Nonlinear Mathematical Physics} \textbf{9} (2002) 347--356.

\bibitem{Vaisman1994poissonbook} I. Vaisman, Lectures on the geometry of Poisson manifolds, \textit{Springer Basel AG,} 1994.

\bibitem{W96} A. Weinstein, The modular automorphism group of a Poisson manifold,
\textit{J. Geom. Phys.} \textbf{23} (1997) 379--394.

\end{thebibliography}
\end{document}